\documentclass[twocolumn]{autart}

\usepackage{graphicx}

\usepackage{times}
\usepackage{amsmath}
\usepackage{amssymb}
\usepackage{algorithm,algorithmic}
\usepackage{subfigure}
\usepackage{euscript}
\usepackage{verbatim}
\usepackage{amssymb}
\usepackage{color}
\usepackage{cancel}
\usepackage{stackrel}

\newtheorem{theorem}{Theorem}
\newtheorem{example}{Example} {\normalfont}{\normalfont}
\newtheorem{problem}{Problem}
\newtheorem{corollary}{Corollary}

{\normalfont}{\normalfont}
\newtheorem{proposition}{Proposition}{\normalfont}{\normalfont}
{\normalfont}{\normalfont}
\newtheorem{remark}{Remark}{\normalfont}{\normalfont}
\newtheorem{definition}{Definition}

{\normalfont}{\normalfont}

{\normalfont}{\normalfont}
 {\normalfont}{\normalfont}
\newenvironment*{proof}{\emph{Proof:}}{}

\newcommand*{\qued}{\hfill\ensuremath{\blacksquare}}

\definecolor{blue}{rgb}{0,0,1}
\definecolor{grayDP}{rgb}{0.4,0.4,0.4}
\newcommand{\DP}[1]{{#1}}
\newcommand{\xhyp}{\mathbf{\rho}}
\newcommand{\xaug}{\tilde{\mathbf{x}}}

\newcommand{\relorder}{\DP{\mathbf{d}}}
\begin{document}

\begin{frontmatter}
\runtitle{A probabilistic interpretation of set-membership filtering: application to polynomial systems through polytopic bounding}

\title{A probabilistic interpretation of set-membership filtering: application to polynomial systems through polytopic bounding\thanksref{footnoteinfo}}

\thanks[footnoteinfo]{This paper was not presented at any IFAC
meeting. Corresponding author Alessio Benavoli. Tel.  0041 58 666 6509}
%Fax +31-40-243-4582.}
%\thanks[footnoteinfo2]{This work has been partly supported by ??.}

\author[IDSIA]{Alessio Benavoli}\ead{alessio@idsia.ch}, % (ead) as shown
\author[IMT]{Dario Piga}\ead{dario.piga@imtlucca.it}, % (ead) as shown

\address[IDSIA]{IDSIA Dalle Molle Institute for Artificial Intelligence SUPSI-USI, Manno, Switzerland.}
\address[IMT]{IMT Institute for Advanced Studies Lucca, Piazza San Francesco 19, 55100 Lucca, Italy.}

                             % keyword list or with the
                                          % help of the Automatica
                  % keyword wizard

\begin{abstract}
{Set-membership estimation is usually formulated in the context of set-valued calculus
and no probabilistic calculations are necessary.}
In this paper, we show that set-membership estimation can  be equivalently formulated in the probabilistic setting by employing sets of probability measures.
\DP{Inference} in set-membership estimation  \DP{is} thus carried out by computing expectations \DP{with respect to}  the updated
set of probability measures $\mathcal{P}$ as in the probabilistic case.
In particular, \DP{it is shown that inference can be performed} by solving a particular semi-infinite linear programming problem,
which is a special case of the truncated moment problem in which only the zero-th order moment is known (i.e., the support).
By writing the dual of the above semi-infinite linear programming problem, \DP{it is  shown} that, if the nonlinearities in the measurement and process equations are \DP{polynomial} and if the bounding sets for initial state, process and measurement noises are described by polynomial inequalities, then an approximation of this semi-infinite linear programming problem can efficiently be obtained by using the theory of sum-of-squares polynomial optimization. We then derive a smart greedy procedure to compute a polytopic outer-approximation of the true membership-set, by computing the  minimum-volume polytope
that outer-bounds the set that includes all the means computed \DP{with respect to}  $\mathcal{P}$.

\end{abstract}

\end{frontmatter}

\section{Introduction} \vspace{-3mm}
Inferring the value of the state of a dynamical system at the various time instants is a classical problem in control and estimation theory. The state is estimated based on noisy signal observations and on a state transition model, which in turn is affected by two sources of uncertainty (namely, process \DP{disturbance} and uncertainty on the initial state conditions).
{
In the literature, there are two main approaches for dealing with the uncertainties and noises acting on the system:
 \begin{itemize}
   \item the \textit{stochastic (probabilistic) approach} that assumes that the noises and the  uncertainties are unknown
but they can be described by known probability distributions.
   \item  the \textit{set-membership approach} that assumes that the noises and the  uncertainties are unknown
but bounded \DP{in} some compact sets.
\end{itemize}
%In both cases, the aim is to infer the value of the state of the dynamic system at the various time instants.
%This problem is solved by different methods.
The probabilistic approach is grounded on  Bayesian filtering, whose aim is
to update with the measurements and propagate up on time the \emph{probability density function} (PDF) of the state.
Inferences are then carried out by computing expectations \DP{with respect to}  this PDF, i.e., mean, variance, credible regions.
It is well known that, for linear discrete-time dynamical systems corrupted by Gaussian noises, \DP{the} Bayesian filter reduces to \DP{the}  Kalman filter.\\
~\\
% In the deterministic setting, the most adopted technique is set-membership estimation.
The set-membership approach is instead based on the construction of a compact set which is guaranteed to include  the
state values of the system that are consistent with the measured output and the assumed  bounds on the noises/disturbances \cite{milvic,Combettes1993,minoplwa,Milanese11,Piga2011set,casinifeas2014}}.
This compact set is propagated \DP{in} time and updated recursively with the output observations.
In set-membership estimation, computing inferences thus means to determine this compact set. Set-membership estimation was first proposed in  \cite{Schweppe67,Bertsekas71}, where an ellipsoidal bounding of the state of linear dynamical systems
is computed. The application of ellipsoidal sets to the state estimation
problem has also been studied by other authors, for example \cite{kuntsevich1992guaranteed,savkin1998robust},
{and, independently, in the communications and signal
processing community, starting from the works  \cite{Fogel1982229,deller1989implementing,DELLER1989301,Deller1994}.}
In order to improve the estimation accuracy,
the use of a convex polytope instead of an ellipsoid has been proposed in \cite{Lahanier89,mo1990fast}.
Unfortunately such a polytope may be extremely complex and the corresponding polytopic updating algorithms
 may require an excessive amount of calculations and storage (without any approximations, the number of vertices of the polytope
increases exponentially in time). For this reason, it has been suggested  to outer approximate the true polytope with a simpler polytope,
i.e. possessing a limited number of vertices or, equivalently, faces \cite{broman1990compact}. In this respect, a parallelotopic approximation of the set-membership set was presented in \cite{chisci1996recursive,chisci1998block}.
A parallelotope is  \DP{the} generalisation of \DP{the}  parallelogram to $\mathbb{R}^n$.
Minimum-volume bounding parallelotopes are then used to
estimate the state of a discrete-linear dynamical system
with polynomial complexity.
Zonotopes have been proposed to reduce the conservativeness of parallelotopes.
Intuitively zonotopes are polytopes with parallel faces, for a more precise
definition see \cite[Ch. 2]{le2013zonotopes}. A  parallelotope is thus a special  zonotope.
Zonotopes are used in  \cite{puig2003worst,combastel2003state,le2011new}
to build a state bounding observer in the context of linear discrete systems.

Zonotopes are also employed to address the problem of  set-membership estimation
for non-linear discrete-time systems with a bounded description of noise and uncertainties \cite{Alamo03}.
At each sample time, a guaranteed bound of the uncertain state trajectory of the system
is calculated using interval arithmetic applied to the nonlinear functions
through the mean interval extension theorem.
% In practice, a guaranteed bound of the range of
% a non-linear function is obtained by means of interval
% arithmetic, i.e., by applying the so called natural interval extension
% or the less conservative mean interval extension \cite{Alamo03}.
This outer bound is represented by a zonotope.
Similar approaches for set-membership  estimation for nonlinear systems are presented in  \cite{calafiore2005reliable,el2001robust,Maier2009}, where
 ellipsoids  are used instead of zonotopes. \DP{Recently, randomized methods are used in \cite{dabbene2015randomized} to approximate, with probabilistic guarantees, the uncertain state trajectory with polynomial sublevel sets.}
% A convex
% optimization problem is formulated to minimize a size criterion for the zonotope that bounds the intersection between
% the uncertain trajectory and the region of the state space that
% is consistent with the measured output.
%A different solution is adopted in \cite{Maier2009}.
%Here the authors derive a set-valued filter for discrete time nonlinear polynomial systems
%using sum-of-squares (SOS) polynomial optimization, again approximating the true membership set $\mathcal{X}$ with an ellipsoid.

% It is also less conservative than the approach presented in \cite{Maier2009}.

The aim of this paper is to address the problem of the estimation of
the state of a discrete-time non-linear dynamical system (characterized by polynomial non-linearities)
in which initial state and noises are unknown but bounded by some compact sets (defined by polynomial inequalities).
We are therefore in the context of set-membership estimation,
but we will address this problem in a very different way
from the approaches presented above.
We reformulate set-membership in the probabilistic setting and solve it using the theory of moments
and positive polynomials.
More precisely the contributions  are the following.

First, by exploiting recent results on filtering with sets of probability measures  
\cite{benavoli2013b,benavoli_2010_journal_b}, we show that set-membership estimation can  be equivalently formulated in 
a probabilistic setting by employing sets of probability measures. In particular, we show that the prediction and 
updating steps of set-membership estimation can be obtained by applying Chapman-Kolmogorov equation and Bayes' rule 
point-wise to the elements of this set of probability measures  $\mathcal{P}$ . This unifies the probabilistic approach 
(Bayes filter) and the set-membership approach to state estimation. {This result can have an enormous impact, because it 
finally can allow us to combine set-membership and classical probabilistic uncertainty in order to obtain hybrid 
filters, i.e., stochastic (probabilistic) filters that are for instance able to use information about the bounding 
region as well as the probabilistic moments (mean and variance) of the noises or that are able to deal with a Gaussian 
measurement noise and a bounded, with known moments, process noise etc..  {Moreover, it can allow us to compute credible
regions (Bayesian confidence intervals)  that takes into account of both deterministic and probabilistic uncertainty, as 
well as it allows us to  make decisions by choosing the action that minimizes the expectation of some loss function 
(this is important, for instance, in control design). 
{ In the context of this paper a first attempt in combining deterministic and probabilistic uncertainty has been 
proposed in \cite{benavoli2013b}, while \cite{Combastel2015289}
has proposed a joint Zonotopic and Gaussian Kalman filter  for discrete-time LTV systems simultaneously subject to bounded disturbances and Gaussian noises.
The work  \cite{Canti6693825} instead proposes a Bayesian approach to set-membership estimation imposing a uniform distribution on
the membership-set similar to the idea proposed in \cite{Gning2010,Gning2011}. We will show that this approach is different from set-membership estimation, since set-membership estimation cannot be interpreted in the Bayesian framework, but only
in the framework of set of probability measures.}}}\\
Second, under this probabilistic interpretation, inferences in set-membership estimation  are carried out by computing expectations \DP{with respect to}  the
set $\mathcal{P}$ as in the probabilistic case. In particular, we show that the membership set $\mathcal{X}$ (i.e., the set that includes the state with guarantee)  can be obtained by computing the union of the supports of the probability measures in $\mathcal{P}$.
Moreover, we prove that a minimum volume convex outer-approximation of  $\mathcal{X}$ can simply be obtained by computing the set $\mathcal{M}$  that includes all the means computed \DP{with respect to}  the probabilities in $\mathcal{P}$.
The proof is not constructive, hence we do not have a convenient description of $\mathcal{M}$. However we show that we can determine the least conservative
half-space $\mathcal{H}$ that includes $M$, by solving a semi-infinite linear programming problem.
This problem is a special case of the truncated moment problem  
\cite{shohat1950problem,kreuin1977markov,lasserre2010moments} in which only the zero-th order moment is known (i.e., the 
support).\\
Third, by writing the dual of the above semi-infinite linear programming problem, we show that, if the nonlinearities in 
the measurement and process equations are polynomial and if the bounding sets for initial state, process and measurement 
noises are described by polynomial inequalities, then
a feasible solution of the dual can be obtained by simply checking the non-negativity  of a polynomial on a compact set described by polynomial inequalities.
An approximation of this semi-infinite linear programming problem can be
obtained by reformulating it as  semidefinite programming by using the theory of \emph{sum-of-squares} (SOS) polynomial optimization.
We prove that the approximate solution is robust, in the sense that the computed half-space $\mathcal{H}$  is guaranteed 
to include  $\mathcal{M}$, and so the membership set  $\mathcal{X}$.\\
Fourth,  we provide a procedure to determine the minimum-volume polytope $\mathcal{S}$  bounding $\mathcal{M}$.  This  procedure is based on a refinement of the algorithm originally proposed   in \cite{piga2012polytopic}  to compute an approximation of the minimum-volume polytope containing a given semialgebraic set.
In particular, we use a Monte Carlo integration approach to compute an approximation of the volume of a polytope,
and a greedy procedure to determine an outer-bounding polytope  $\mathcal{S}$ as the intersection of a pre-specified number of half-spaces $\mathcal{H}_j$, where each half-space $\mathcal{H}_j$ is added to the description of $\mathcal{S}$ so to minimize the volume of the polytope including $\mathcal{M}$. 
This allows us to solve the set-membership estimation problem for polynomial non-linear systems very efficiently and through convex optimization. % Therefore, for polynomial systems, it is less conservative than the approach proposed in \cite{Alamo03}.
%It is also less conservative than the approach presented in \cite{Maier2009}, which uses SOS optimization but employs an ellipsoid
%to bound the set $\mathcal{X}$ and ellipsoids are in general more conservative than polytopes.
\\ Finally, by means of a numerical example involving the Lotka Volterra prey-predator model, we show the effectiveness 
of our approach.

% \cite{Maier2009}

% \begin{itemize}
%   \item Filtraggio Set-Membership Non lineare (polinomiale)
%   \item Equazioni e vincoli polinomiali (es. vincoli su energia di w e v, o vincoli norma 1 o norm infinito)
%   \item Calcolo Politopo minima area contenente traiettorie dello stato
%   \item Tencnia utilizzata SOS
% \end{itemize}

\section{Problem Description} \label{Sec:probdes}\vspace{-1mm}
Consider an uncertain non-linear discrete-time  dynamical system described by the difference equations:
\begin{equation}
\label{eq:setmembersys} \left\{
\begin{array}[pos]{rcl}
   \mathbf{x}(k)&=& \mathbf{a}_d(\mathbf{x}(k-1),k-1)+\mathbf{w}(k-1),\\
   \mathbf{y}(k)&=&\mathbf{c}_d(\mathbf{x}(k),k)+\mathbf{v}(k),\\
\end{array}\right.
\end{equation}
where $\mathbf{x}(k)=[x_1(k),\dots,x_n(k)]^\top\in \mathbb{R}^n$ is the state of the system  at the time $k$,
$\mathbf{y}(k)\in \mathbb{R}^m$ is the measured output vector, $\mathbf{w}(k-1) \in \mathbb{R}^n$ is the
process noise and $\mathbf{v}(k) \in \mathbb{R}^m$ is the measurement noise.
In this paper, we consider polynomial non-linearities $\mathbf{a}_d(\mathbf{x}(k),k)$ and $\mathbf{c}_d(\mathbf{x}(k),k)$, i.e.,
\begin{subequations} \label{eq:setmembersysa1tot}
\begin{align}
 \mathbf{a}_d(\mathbf{x}(k-1),k-1)=&\mathbf{A}_{k-1} \mathbf{q}_d(\mathbf{x}(k-1)), \label{eq:setmembersysa1} \\
 \mathbf{c}_d(\mathbf{x}(k),k)=&\mathbf{C}_k \mathbf{q}_d(\mathbf{x}(k)), \label{eq:setmembersysa2}
 \end{align}
\end{subequations}
with
\begin{equation}
\label{eq:moments}
\begin{array}{l}
 \mathbf{q}_d(\mathbf{x})= \\
~[1,x_1,\dots,x_n,x_1^2,x_1x_2,\dots,x_{n-1}x_{n},x_n^2,\dots,x_1^d,\dots,x_n^d]^\top
 \end{array}
\end{equation}
being the vector of all monomials of degrees less than or equal to $d$, which has dimension $s(d)={n+d \choose d} $,
and $\mathbf{A}_{k-1} \in \mathbb{R}^{n\times s(d)}$, $\mathbf{C}_k \in \mathbb{R}^{m\times s(d)}$
are known time-variant coefficient  matrices. The resulting system will be referred in the paper as uncertain time-variant polynomial system
of degree $d$.
{
\begin{example}
 \label{ex:0}
Let us consider the discrete-time polynomial system:
\begin{align}
\nonumber
 x_1(k)&= x_1(k-1)\left(2-x_1(k-1)\right)+w_1(k-1),\\
   %x_1(k)&=& (r+1)x_1(k-1)-rx^2_1(k-1)-bx_1(k-1)x_2(k-1)+w_1(k-1),\\
   \nonumber
   x_2(k)&= x_1(k-1)x_2(k-1)+0.5x_2(k-1)+w_2(k-1),
\end{align}
%with initial conditions $\mathbf{x}(0)=[x_1(0) \ \ x_2(0)]^\top=[0.1 \ \ -0.1]^\top$.
The output equation is given by:
${y}(k)=x_1(k)+x_2(k)+{v}(k)$. 
We can rewrite this system as in (\ref{eq:setmembersysa1})--(\ref{eq:setmembersysa2}):
\begin{align}
\nonumber
 \mathbf{q}_d(\mathbf{x})&=[1,x_1,x_2,x_1^2,x_1x_2,x_2^2]^\top\\
\nonumber
 \mathbf{A}_{k-1}&=\left[
\begin{matrix}
 0&2&0&-1&0&0\\
 0&0&0.5&0&1&0\\
\end{matrix}
\right]\\
\nonumber
\mathbf{C}_{k-1}&=\left[
\begin{matrix}
 0&1&1&0&0&0\\
 \end{matrix}
\right]
\end{align}
and therefore $n=2$, $m=1$, $d=2$ and $s(d)={n+d \choose d}=6$.
\end{example}}

We further assume that the only available information about the initial state
$\mathbf{x}(0)$ and the noises $\mathbf{w}(k),\mathbf{v}(k)$ is:
\begin{equation}
\label{eq:sm}
\begin{array}{l}
\mathbf{x}(0)\in \mathcal{X}_0, ~~\mathbf{w}(k)\in \mathcal{W}_k,~~\mathbf{v}(k)\in \mathcal{V}_k,
 \end{array}
\end{equation}
where $\mathcal{X}_0,\mathcal{W}_k,\mathcal{V}_k$ are compact basic semi-algebraic sets, i.e.,
compact sets  described by the polynomial inequalities:
\begin{equation}
\label{eq:semi}
\begin{array}{l}
\mathcal{W}_k=\left\{\mathbf{w}(k)\in \mathbb{R}^n: ~h^w_i(\mathbf{w}(k),k)\leq0, ~i=1,\dots,t_w\right\},
 \end{array}
\end{equation}
where $h^w_i$ (with $i=1,\dots,t_w$, $t_w \in \mathbb{N}$) are polynomial functions in the variable $\mathbf{w}(k)$. The  sets
$\mathcal{X}_0,\mathcal{V}_k$ are described in a similar manner.

This paper addresses a set-membership filtering problem, which aims at recursively estimating, at each time sample $k=1,2,\ldots,T_{\mathrm{o}}$, (an outer approximation of) the state uncertainty set $\mathcal{X}_k$, defined as the set of all values $\mathbf{x}(k)$ compatible with the available information, namely the system equations \eqref{eq:setmembersys}, the bounds on the initial state and on the noises  \eqref{eq:sm}, and the output observations $\mathbf{y}(1),\mathbf{y}(2),\ldots,\mathbf{y}(T_{\mathrm{o}})$. Formally, the set-membership filtering problem is defined as follows.

 \begin{problem}  \normalfont \textbf{[Set-membership filtering] \label{Problem:PWAregr}} \\
 Given the system equations \eqref{eq:setmembersys}, the observations, the bounding sets for the noises $\mathcal{W}_k,\mathcal{V}_k$ and the initial state uncertainty set $\mathcal{X}_0$, compute recursively the state uncertainty set $\mathcal{X}_k$ defined as:
 \begin{equation} \label{eqn:defXk}
 \begin{array}{ll}
 \mathcal{X}_k=\left\{ \right.  \mathbf{x}(k)\! \in \!\mathbb{R}^n\!\!: &   \mathbf{x}(k) \! - \! \mathbf{a}_d(\mathbf{x}(k-1),k-1) \! \in \! \DP{\mathcal{W}_{k-1},} \nonumber  \\
   &  \mathbf{y}(k)-\mathbf{c}_d(\mathbf{x}(k),k) \in \mathcal{V}_k, \nonumber \\
   &  \mathbf{x}(k-1) \in \mathcal{X}_{k-1}   \left. \right\}
 \end{array}
 \end{equation}
   for each $k=1,2,\ldots,T_{\mathrm{o}}$. \hfill $\blacksquare$
 \end{problem}
 Note that, in general, the sets $\mathcal{X}_k$ might be nonconvex and their representation can become more and more 
complicated as the time index $k$ increases. Under the assumption that $\mathcal{X}_k$ is bounded, algorithms for 
computing simple sets (e.g., boxes, parallelotopes, zonotops or ellipsoidal regions) outer-bounding the state 
uncertainty sets $\mathcal{X}_k$  have been then proposed to reduce this complexity. %in the literature \cite{}.
After formulating the set-membership filtering problem in a probabilistic setting, this paper presents an algorithm for computing (an approximation of) the minimum-volume polytope outer-bounding the  sets $\mathcal{X}_k$.
%
% \DP{\begin{remark}
% Without loss of generality, autonomous systems not driven by external inputs are considered in \eqref{eq:setmembersysa1tot}.  Nevertheless,  the presence of known input signals can be also handled (provided that the equations describing the system in \eqref{eq:setmembersysa1tot} remain polynomial functions in the state $\mathbf{x}(k)$), as their presence does not change the structure of the set-membership filtering problem formulated above. Specifically, the values of the time-varying matrices $\mathbf{A}_{k-1}$ and $\mathbf{C}_k$ will be affected by the input signal measurements, and a known time-varying term will appear in eq. \eqref{eq:setmembersysa1tot}.
% \end{remark}}

%%%%%%%%%%%%%%%%%%%%%%%%%%%%%%%%%%%%%%%%%%%%%%%%%%%%%%%%%%%%%%%%%%%%%%%%%%%%%%%%

\section{A probabilistic framework for set-membership estimation}
 Set-membership estimation is usually formulated in the context of set-valued calculus.
% \DP{This is not completely true; in \cite{} it has been shown  that set-membership estimation can equivalently be formulated in the stochastic setting by employing sets
% of probability measures. Commento: toglierei questa frase... rischia di sminuire i risultati presentati in questa sezione. Metterei la frase sotto}.
We will show in the following paragraph that  set-membership estimation can  be equivalently formulated in the probabilistic setting by employing sets
of probability measures. %\subsection{Theory-of-moments formulation}
Consider the set-membership constraint $\mathbf{x}\in \mathcal{X}$ (the time index is dropped for brevity of notation)
with $\mathcal{X}\subset \mathbb{R}^n$. This constraint can be translated in a probabilistic setting
by saying that the only probabilistic information on the value $\mathbf{x}$ of the variable $\mathbf{X}$ is
that it belongs to the set $\mathcal{X}$, or equivalently,
$$
\Pr(X \in \mathcal{X})=\Pr(\mathcal{X})=1,
$$
where $\Pr$ is a probability measure on $\mathcal{X}$.\footnote{To clarify this aspect, consider the experiment of rolling a  dice. Assume that the probability  $\Pr$ of the outcomes $x$ of the dice is completely unknown, then the only knowledge about the experiment
is that $x \in \mathcal{X}=\{1,2,3,4,5,6\}$, or, equivalently, that $\Pr(\{1,2,3,4,5,6\})=1$. Therefore, the statement $\Pr(\mathcal{X})=1$
is a model for our (epistemic) uncertainty about the probabilities of the dice outcomes. We only know that $x \in \{1,2,3,4,5,6\}$.}
More precisely $\Pr$ is a nonnegative Borel measure on $\mathcal{X}$.\footnote{The sample space is $\mathbb{R}^n$ and we are considering the Borel
$\sigma$-algebra. $\mathcal{X}$ is assumed to be an element of the $\sigma$-algebra.}
In other words, this means that we only know the support of the probability measure of the variable $\mathbf{X}$.

The support does not uniquely define a probability measure, as  there  \DP{are an indefinite number of probability measures} with support $\mathcal{X}$.\footnote{The uniform distribution is one of them, but it is not the only one. So by considering only the uniform distribution as in \cite{Canti6693825}, we loose the full equivalence with set-membership.} Hence,
  $\mathbf{x}\in \mathcal{X}$ is equivalent to the constraint
that the probability measure of $\mathbf{x}$ belongs to the set  $\mathcal{P}_{\mathcal{X}}(\mathbf{X})$,
that is the set of all probability measures on the variable $\mathbf{X}$  with support $ \mathcal{X}$.
Let us define with $P$ the Cumulative Distribution Function (CDF) of the probability measure $\Pr$.
For instance on $\mathbb{R}$ we have that $P(x)=\Pr(-\infty,x]$ (this definition can easily be extended to $\mathbb{R}^n$).
Then we can easily characterize the set of probability measures $\mathcal{P}_{\mathcal{X}}(\mathbf{X})$ as follows:
\begin{equation}
\label{eq:probconstr0}
\begin{array}{l}
\mathcal{P}_{\mathcal{X}}(\mathbf{X})=\left\{P: ~~\int_{\mathcal{X}} dP(\mathbf{x})=1\right\},
 \end{array}
\end{equation}
where the integral is a Lebesgue-Stieltjes integral \DP{with respect to}  $P$. Hence, because
of the equivalence between Borel probability measures and cumulative distributions,
hereafter we will use  interchangeably $\Pr$ and $P$.

\subsection{Inference on the state}
In state estimation, we are interested in making inferences about $\mathbf{X}$
or, equivalently, computing expectations of real-valued functions $g$ of $\mathbf{X}$.
Since there \DP{are an indefinite number of probability measures} with support $ \mathcal{X}$, we cannot compute
a single expectation of $g$. However, we can compute upper and lower \DP{bounds} for the expectation of $g$ \DP{with respect to}  the probability measures $\Pr$ with support $ \mathcal{X}$. For instance, the upper bound for the  expectation of $g$ is given by the solution of the optimization problem:
 \begin{equation}
\label{eq:probconstroptim}
\begin{array}{l}
\sup\limits_{P} \int_{\mathcal{X}} g(\mathbf{x}) dP(\mathbf{x}),\\
s.t.~~P \in \mathcal{P}_{\mathcal{X}}(\mathbf{X}),
 \end{array}
\end{equation}
which is a semi-infinite linear program, since it has a finite number constraints
and an infinite dimensional variable (the probability measure $\Pr$).
Note that we use ``\emph{sup}'' instead of ``$\emph{max}$'' to indicate that an optimal solution might not be attained.
The lower bound of the expectation can be obtained by replacing \emph{sup} with \emph{inf}.

Problem (\ref{eq:probconstroptim}), i.e.,
determining an upper bound for the expectation  of $g$ \DP{with respect to}  the probability measure $\Pr$ given the knowledge of its support $ \mathcal{X}$, is a special case
of the truncated moment problem  \cite{shohat1950problem,kreuin1977markov,lasserre2010moments} in which only the zero-th order moment is known (i.e., the support).
Hence, we have the following result \cite{karr1983extreme},  \cite[Lemma 3.1]{shapiro2001duality}:

\begin{proposition}
\label{prop:1}
The optimum of  (\ref{eq:probconstroptim}) is obtained by an atomic measure\footnote{An atomic measure in $\mathbb{R}^n$ is a measure which accepts as an argument a subset $A$ of $\mathbb{R}^n$, and returns $\delta_{\mathbf{x}}(A) = 1$ if $\mathbf{x} \in A$,  zero otherwise.} $\Pr=\delta_{\hat{\mathbf{x}}}$, where
$\hat{\mathbf{x}} =\arg\sup_{\mathbf{x} \in \mathcal{X}} g(\mathbf{x})$.
\end{proposition}
Note in fact that, $\forall \Pr \in \mathcal{P}_{\mathcal{X}}(\mathbf{X})$, with associated CDF $P$,
$$
E[g]=\int_{\mathcal{X}} g(\mathbf{x})dP(\mathbf{x})\leq \int_{\mathcal{X}} g(\mathbf{x})\delta_{\hat{\mathbf{x}}}(d\mathbf{x})=g(\hat{\mathbf{x}}),
$$
where $g(\hat{\mathbf{x}})$, by definition of $\hat{\mathbf{x}}$, is the supremum   of $g$ on $\mathcal{X}$.
  The first integral must be understood as a Lebesgue-Stieltjes integral \DP{with respect to}  the cumulative
  distribution of an atomic measure on $\mathbb{R}^n$. This means that $dP(\mathbf{x})$ denotes
  the distributional derivative of the cumulative distribution of an atomic measure,  that are in our case Dirac measures $\delta_{\hat{\mathbf{x}}}(d\mathbf{x})$ (hence the second integral).
From this result, it follows \DP{that the} probability measures that
gives the lower and upper bounds for the expectation of $g$ are atomic (discrete) measures.\\
In order to formulate the set-membership filtering problem in a probabilistic framework \DP{it} is useful to exploit a result
derived by Karr in \cite{karr1983extreme}, where it is  proven that  the set of probability measures  $\mathcal{P}_{\mathcal{X}}(\mathbf{X})$ which are feasible for the semi-infinite linear program
problem (\ref{eq:probconstroptim}) is convex and compact with respect to the weak$^*$ topology.
As a result, $\mathcal{P}_{\mathcal{X}}(\mathbf{X})$  can be expressed as the convex hull of its extreme points and, according to Proposition \ref{prop:1}, these extreme points are
atomic measures on $\mathcal{X}$, i.e.:
 \begin{equation}
\label{eq:probconstroptimequiv}
\begin{array}{l}
\mathcal{P}_{\mathcal{X}}(\mathbf{X})\equiv Co\left\{\delta_{\hat{\mathbf{x}}}:~\hat{\mathbf{x}}\in  \mathcal{X}\right\},
 \end{array}
\end{equation}
where $\equiv$ means equivalent in terms of inferences (expectations).
% Here, we are introducing an abuse of notation:  $\delta_{\hat{\mathbf{x}}}(\mathbf{x})$  will denote for us the Kronecker delta,\footnote{The function is $1$ if $\mathbf{x}=\hat{\mathbf{x}}$, and $0$ otherwise.} while we will use $d\delta_{\hat{\mathbf{x}}}(\mathbf{x})$ to denote its distributional derivative (its ``probability density function'', that is a Dirac's delta).
Summing up what we have obtained so far:
\begin{enumerate}
 \item the set-membership constraint $\mathbf{x} \in \mathcal{X}$  is equivalent to (\ref{eq:probconstr0});
 \item for the inferences, $\mathcal{P}_{\mathcal{X}}(\mathbf{X})$ is equivalent to the convex hull of all atomic measures on $\mathcal{X}$, \eqref{eq:probconstroptimequiv}.
\end{enumerate}
Hence, we can derive the prediction and updating step for set-membership estimation by applying  the
Chapman-Kolmogorov equation and Bayes' rule to the set of probability measures  in  \eqref{eq:probconstroptimequiv}.
This means that, by reformulating set-membership constraints in a probabilistic way, we can reformulate set-membership estimation in the realm of stochastic (probabilistic) filtering applied to set of probability measures.

% Equation \eqref{eq:probconstroptimequiv} is used to propagate and update in time a set of
% probability distributions. In other words

\subsection{Propagating in time and updating  set of distributions}
We start by deriving the set-membership filtering prediction step by applying  the
Chapman-Kolmogorov equation.

\begin{theorem}[Prediction]
\label{Th:prediction}
Consider the system equation in (\ref{eq:setmembersys}) with $\mathbf{w}(k-1) \in \mathcal{W}_{k-1}$ and assume
 that the only  probabilistic knowledge about $\mathbf{X}(k-1)$ is the support $\mathcal{X}_{k-1}$.
 Then it follows that the probability measure $\Pr$ on the value $\mathbf{x}(k)$ of the state at time $k$ belongs to the set
 \begin{equation}
\label{eq:probconstr}
\begin{array}{l}
\hat{\mathcal{P}}_{{\hat{\mathcal{X}}_k}}(\mathbf{X}(k))\equiv Co\left\{\delta_{\hat{\mathbf{x}}}:~\hat{\mathbf{x}}\in  {\hat{\mathcal{X}}_k}\right\}, %\left\{P: ~~\int_{\mathcal{X}_{k}} dP(\mathbf{x}(k))=1\right\},
 \end{array}
\end{equation}
with
 \begin{align}
\nonumber
{\hat{\mathcal{X}}_k}=\Big\{\mathbf{x}(k): \mathbf{x}(k)=\mathbf{a}_d(\mathbf{x}(k-1),k-1)+\mathbf{w}(k-1)\\
\label{eq:predictionv}
~with~ \mathbf{x}(k-1) \in \mathcal{X}_{k-1}, ~\mathbf{w}(k-1) \in \mathcal{W}_{k-1}\Big\},
\end{align}
or equivalently:
 \begin{align}
\nonumber
{\hat{\mathcal{X}}_k}=\Big\{\mathbf{x}(k): \mathbf{x}(k)-\mathbf{a}_d(\mathbf{x}(k-1),k-1) \in \mathcal{W}_{k-1} \\
\label{eq:predictionv2}
~with~ \mathbf{x}(k-1) \in \mathcal{X}_{k-1}\Big\}.
\end{align}
\end{theorem}
\begin{proof}
Let us consider the time instant $k$.  From the system equation in (\ref{eq:setmembersys}), $\mathbf{w}(k-1) \in \mathcal{W}_{k-1}$ and \eqref{eq:probconstroptimequiv}, it follows that
 $$
\begin{array}{l}
\mathcal{P}(\mathbf{X}(k)|\mathbf{x}(k-1))\\
\equiv Co\left\{\delta_{\mathbf{a}_d(\mathbf{x}(k-1),k-1)+\mathbf{\hat{w}}}:~~\mathbf{\hat{w}} \in \mathcal{W}_{k-1}\right\},\end{array}
 $$
 this is the conditional set of probability measures for the variable $\mathbf{X}(k)$ given the value $\mathbf{x}(k-1)$ of the variable
 $\mathbf{X}(k-1)$
Hence, since $\mathbf{X}(k-1) \in \mathcal{X}_{k-1}$ and so the set of probability measures for the variable $\mathbf{X}(k-1)$ is
$$
\begin{array}{l}
\mathcal{P}_{\mathcal{X}_{k-1}}(\mathbf{X}(k-1))\equiv Co\left\{\delta_{\hat{\mathbf{x}}}:~\hat{\mathbf{x}}\in  \mathcal{X}_{k-1}\right\},
 \end{array}
$$
  by applying the Chapman-Kolmogorov equation point-wise to the probability measures  $\Pr(\cdot|\mathbf{x}(k-1))$ in $\mathcal{P}(\mathbf{X}(k)|\mathbf{X}(k-1))$ and  $\text{Qr}$    in $\mathcal{P}(\mathbf{X}(k-1))$
  we obtain
 \begin{align} \label{eq:Px}
 \nonumber
  &\Pr(\mathbf{x}(k))=\int_{\mathbb{R}^n} \int_{\mathbb{R}^n} I_{\mathbf{x}(k)}(\mathbf{x}') dP(\mathbf{x}'|\mathbf{x}(k-1))dQ(\mathbf{x}(k-1))\\
 \nonumber
 &=\int_{\mathbb{R}^n} \int_{\mathbb{R}^n} I_{\mathbf{x}(k)}(\mathbf{x}')  \delta_{\mathbf{a}_d(\mathbf{x}(k-1),k-1)+\mathbf{\hat{w}}}(d\mathbf{x}') \delta_{\hat{\mathbf{x}}}(d\mathbf{x}(k-1))\\
  \nonumber
  &=\int_{\mathbb{R}^n} \delta_{\mathbf{a}_d(\mathbf{x}(k-1),k-1)+\mathbf{\hat{w}}}(\mathbf{x}(k)) \delta_{\hat{\mathbf{x}}}(d\mathbf{x}(k-1))\\
  \nonumber
  &=\delta_{\mathbf{a}_d(\hat{\mathbf{x}},k-1)+\mathbf{\hat{w}}}(\mathbf{x}(k))\\
  \end{align}
   where $I_{\mathbf{x}(k)}(\mathbf{x}')$ denotes the indicator function\footnote{$I_{\mathbf{x}(k)}(\mathbf{x}')=1$ when $\mathbf{x}(k)=\mathbf{x}'$ and zero otherwise.}   and with $\hat{\mathbf{x}} \in \mathcal{X}_{k-1}$ and $\mathbf{\hat{w}} \in \mathcal{W}_{k-1}$ and where we have exploited the fact that
   $ \int_{\mathbb{R}^n} I_{\mathbf{x}(k)}(\mathbf{x}')  \delta_{\mathbf{a}_d(\mathbf{x}(k-1),k-1)+\mathbf{\hat{w}}}(d\mathbf{x}')=I_{\mathbf{x}(k)}(\mathbf{a}_d(\mathbf{x}(k-1),k-1)+\mathbf{\hat{w}})=\delta_{\mathbf{a}_d(\mathbf{x}(k-1),k-1)+\mathbf{\hat{w}}}(\mathbf{x}(k))$.
     From \eqref{eq:probconstroptimequiv}, \eqref{eq:Px} and the definition of ${\hat{\mathcal{X}}_k}$, the theorem follows. \qued\\
   \end{proof}
   ~\\
   {
Theorem \ref{Th:prediction} shows that, by applying  the Chapman-Kolmogorov equation point-wise to the probability measures in $\mathcal{P}(\mathbf{X}(k)|\mathbf{x}(k-1))$ and $\mathcal{P}_{\mathcal{X}_{k-1}}(\mathbf{X}(k-1))$, we can obtain a set of probability measures $\mathcal{P}_{\hat{\mathcal{X}}_k}(\mathbf{X}(k))$,
which is completely defined by its support and whose support coincides with the one obtained in set-membership estimation after the prediction step.}

We now derive a similar result for  the updating step.

\begin{theorem}[Updating] \label{Th:updating}
 Consider the measurement equation in (\ref{eq:setmembersys}) with $\mathbf{v}(k) \in \mathcal{V}_k$ and assume
 that the only  probabilistic knowledge about $\mathbf{x}(k)$ is described by (\ref{eq:probconstr})--(\ref{eq:predictionv}).
 Then it follows that the updated probability measures $P$ on the value $\mathbf{x}(k)$ of the state at time $k$ belongs to the set:
 \begin{equation}
\label{eq:probconstr1}
\begin{array}{l}
\mathcal{P}_{\mathcal{X}_{k}}(\mathbf{X}(k))\equiv Co\left\{\delta_{\hat{\mathbf{x}}}:~\hat{\mathbf{x}}\in  {\mathcal{{X}}_k}\right\}, %\left\{P: ~~\int_{\mathcal{\hat{X}}_{k}} dP(\mathbf{x}(k)|\mathbf{y}(k))=1\right\},
 \end{array}
\end{equation}
where
 \begin{align} \label{eq:defXk2}
{\mathcal{X}_{k}}&={\hat{\mathcal{X}}_{k}} \cap \mathcal{Y}_{k},
%\label{eq:updating}
\end{align}
with
\begin{align} \label{def:Yk}
 \mathcal{Y}_{k}&=\{\mathbf{x}(k): \mathbf{y}(k)-\mathbf{c}_d(\mathbf{x}(k),k)\in \mathcal{V}_k\}.
 \end{align}
\end{theorem}
\begin{proof}
Observe that, at each time $k$,
$$
\mathcal{P}(\mathbf{Y}(k)|\mathbf{x}(k))\equiv Co\left\{\delta_{\mathbf{c}_d(\mathbf{x}(k),k)+\mathbf{\hat{v}}}:~\mathbf{\hat{v}}\in \mathcal{V}_k\right\}.
$$
% Hence, for each $\mathbf{\hat{v}}\in \mathcal{V}_k$, the likelihood of the observation $\mathbf{y}(k)$ is $P(\mathbf{Y}(k)=\mathbf{y}(k)|\mathbf{x}(k))=\delta_{\mathbf{c}_d(\mathbf{x}(k),k)+\mathbf{\hat{v}}}(\mathbf{y}(k))$.
Then, the updating step consists of applying Bayes'rule to the probability measures  $\mathcal{P}(\mathbf{Y}(k)|\mathbf{x}(k))$  and to
$Q$    in $\mathcal{\hat{P}}_{\hat{\mathcal{X}}_{k}}(\mathbf{X}(k))$:
\begin{align}
\nonumber
dP(\mathbf{x}(k)|\mathbf{y}(k))&=\dfrac{\int_{\mathbb{R}^m} I_{\mathbf{y}(k)}(\mathbf{y}')dP(\mathbf{y}'|\mathbf{x}(k))dQ(\mathbf{x}(k)) }{\int_{\mathbb{R}^n} \int_{\mathbb{R}^m} I_{\mathbf{y}(k)}(\mathbf{y}')dP(\mathbf{y}'|\mathbf{x}(k))dQ(\mathbf{x}(k)) }\\
\nonumber
&=\dfrac{\Pr(\mathbf{y}(k)|\mathbf{x}(k))dQ(\mathbf{x}(k)) }{\int_{\mathbb{R}^n} \Pr(\mathbf{y}(k)|\mathbf{x}(k)) dQ(\mathbf{x}(k)) },
\end{align}
where we have exploited the fact that
$$
\Pr(\mathbf{y}(k)|\mathbf{x}(k))=\int_{\mathbb{R}^m} I_{\mathbf{y}(k)}(\mathbf{y}')dP(\mathbf{y}'|\mathbf{x}(k)).
$$
Note that the probability of a point on $\mathbb{R}^n$ can be nonzero since $\Pr$ is an atomic measure.
In order to apply Bayes' rule we need to ensure that the denominator is strictly greater than zero:
$$
\begin{array}{l}
\int_{\mathbb{R}^n} \Pr(\mathbf{y}(k)|\mathbf{x}(k))dQ(\mathbf{x}(k))\\
=\int_{\mathbb{R}^n} \delta_{\mathbf{c}_d(\mathbf{x}(k),k)+\mathbf{\hat{v}}}(\mathbf{y}(k))\delta_{\hat{\mathbf{x}}}(d\mathbf{x}(k))>0.
\end{array}
$$
Hence, the above inequality holds if and only if $\hat{\mathbf{x}}$ and $\mathbf{\hat{v}}$ are chosen, \DP{at time $k$,}  such that:
\begin{equation}
\label{eq:consdenom}
\mathbf{c}_d(\hat{\mathbf{x}},k)+\mathbf{\hat{v}}=\mathbf{y}(k).
\end{equation}
Bayes' rule is only defined for those probability measures for which
the denominator is strictly positive, that implies that the above equality must be satisfied.\footnote{This way of updating set of probability measures has been proposed by Walley  \cite[Appendix J]{walley1991} under
the name of regular extension.}
The  equality  \eqref{eq:consdenom} can be satisfied only if $\hat{\mathbf{x}} \in \mathcal{Y}_{k}$ which, together with the constraint $\hat{\mathbf{x}}\in \mathcal{\hat{X}}_k$,
implies that
$$
\hat{\mathbf{x}} \in {\hat{\mathcal{X}}_{k}} \cap \mathcal{Y}_{k}.
$$
Under the constraint \eqref{eq:consdenom},  it follows that $\delta_{\mathbf{c}_d(\hat{\mathbf{x}},k)+\mathbf{\hat{v}}}(\mathbf{y}(k))=1$ and, thus,  the denominator is equal to one.
Hence, we have that
$$
\begin{array}{l}
dP(\mathbf{x}(k)|\mathbf{y}(k))=\int_{\mathbb{R}^m} I_{\mathbf{y}(k)}(\mathbf{y}')dP(\mathbf{y}'|\mathbf{x}(k))dQ(\mathbf{x}(k))\\
=\int_{\mathbb{R}^m} I_{\mathbf{y}(k)}(\mathbf{y}')\delta_{\mathbf{c}_d(\mathbf{x}(k),k)+\mathbf{\hat{v}}}(\mathbf{y}(k))\delta_{\hat{\mathbf{x}}}(d\mathbf{x}(k))\\
=\delta_{\mathbf{c}_d(\mathbf{x}(k),k)+\mathbf{\hat{v}}}(\mathbf{y}(k))\delta_{\hat{\mathbf{x}}}(d\mathbf{x}(k))\\
=\delta_{\mathbf{c}_d(\hat{\mathbf{x}},k)+\mathbf{\hat{v}}}(\mathbf{y}(k))\delta_{\hat{\mathbf{x}}}(d\mathbf{x}(k))\\
=\delta_{\hat{\mathbf{x}}}(d\mathbf{x}(k))
\end{array}
$$
with $\hat{\mathbf{x}} \in {\hat{\mathcal{X}}_{k}} \cap \mathcal{Y}_{k}$. Hence, the updated probability measure $\Pr(\cdot|\mathbf{y}(k))$
on the values of the state at time $k$ is $\Pr(\cdot|\mathbf{y}(k))= \delta_{\hat{\mathbf{x}}}$, which proves the theorem. \qued
\end{proof}

From Theorem \ref{Th:updating}, the support of the updated probability measure $\Pr$ on the value $\mathbf{x}(k)$ of the state at time $k$ is given by ${\mathcal{X}_{k}}$, i.e.,
 \begin{equation}
 \int_{{\mathcal{X}_{k}}}dP(\mathbf{x}(k))=1,
 \end{equation}
 where ${\mathcal{X}_{k}}$  is given by \eqref{eq:defXk2}, or  equivalently by \eqref{eqn:defXk}. In other words,  the support of the probability measure  $\Pr$ of the value of the state $\mathbf{x}(k)$ given the output observation $\mathbf{y}(k)$ and the system equations \eqref{eq:setmembersys} is nothing but ${\mathcal{X}_{k}}$.    This is in accordance with the set-membership formulation, which claims that  $\mathbf{x}(k)$ belongs to state uncertainty set  ${\mathcal{X}_{k}}$ defined in \eqref{eqn:defXk}.
 Then we can solve set-membership filtering by applying recursively  Theorems \ref{Th:prediction} and \ref{Th:updating}, as described in Algorithm 1.
 \begin{algorithm} \emph{Algorithm 1: prediction and updating}  \\
\begin{description}
  \item[A1.1] Initialize $\mathcal{P}_{\mathcal{X}_0}(\mathbf{X}(0))\equiv Co\left\{\delta_{\hat{\mathbf{x}}}:~\hat{\mathbf{x}}\in  \mathcal{X}_0\right\}$.
  \item[A1.2] For $k=1,\dots,T_o$:
  \begin{description}
  \item[A1.2.1]  $\hat{\mathcal{P}}_{\hat{\mathcal{X}}_k}(\mathbf{X}(k))\equiv Co\left\{\delta_{\hat{\mathbf{x}}}:~\hat{\mathbf{x}}\in  {\hat{\mathcal{X}}_k}\right\}$ with
  $\hat{\mathcal{X}}_k$ defined in \eqref{eq:predictionv};
    \item[A1.2.2]  $\mathcal{P}_{\mathcal{X}_k}(\mathbf{X}(k))\equiv Co\left\{\delta_{\hat{\mathbf{x}}}:~\hat{\mathbf{x}}\in  \mathcal{X}_k\right\}$ with
  $\mathcal{X}_k$ defined in \eqref{eq:defXk2}.
%    \item[A.2.3] redefine $\mathcal{P}(\mathbf{x}(k)|\mathbf{y}(k))$ as  $\mathcal{P}(\mathbf{x}(k))$ and iterate.
  \end{description}
\end{description}
%\hfill $\blacksquare$
%\hfill $\square$
\end{algorithm}
The steps A1.2.1 and A1.2.1 are  the prediction and the updating steps, respectively.
Note that the set of probability measures  $\mathcal{P}_{\mathcal{X}_k}(\mathbf{X}(k))$ (or 
$\hat{\mathcal{P}}_{\hat{\mathcal{X}}_k}(\mathbf{X}(k))$) is computed by taking into account all the observations  
$\mathbf{y}^{k}=\{\mathbf{y}(1),\mathbf{y}(2),\dots,\mathbf{y}(k)\}$ (respectively $\mathbf{y}^{k-1}$).
Hence, it should be more correctly denoted as $\mathcal{P}_{\mathcal{X}_k}(\mathbf{X}(k)|\mathbf{y}^{k})$ (respectively $\mathcal{P}_{\hat{\mathcal{X}}_k}(\mathbf{X}(k)|\mathbf{y}^{k-1})$). We have omitted this  notation for brevity.

\begin{remark}
 Under the assumptions \eqref{eq:setmembersysa1},\eqref{eq:setmembersysa2} and \eqref{eq:semi}, the set $\mathcal{X}_k$ is a semialgebraic set in $\mathbb{R}^{n}$,  described by the intersections of the semialgebraic sets $\hat{\mathcal{X}}_k$ (Eq. \eqref{eq:predictionv2}) and $\mathcal{Y}_{k}$ (Eq. \eqref{def:Yk}). Formally,  $\mathcal{X}_k$ is the projection in the space of $\mathbf{x}(k)$  of the set
 \begin{align} \label{def:Xkt}
 \tilde{\mathcal{X}}_k=\left\{\xaug \in \mathbb{R}^{2n}: h_s(\xaug (k)) \leq 0, \ s=1,\ldots,m\right\},
 \end{align}
  where $\xaug (k)$ is the augmented state vector $\xaug (k)=\left[\mathbf{x}^\top(k) \ \ \mathbf{x}^\top(k-1)\right]^\top$ and $h_s(\xaug (k))$ (with $s=1,\ldots,m$) are the polynomial functions in  $\mathbf{x}(k)$ and  $\mathbf{x}(k-1)$ (or equivalently in $\xaug (k)$) defining $\hat{\mathcal{X}}_k$  and $\mathcal{Y}_{k}$.
  In the rest of the paper, we will use the following   notation to describe the set $\mathcal{X}_k$:
 \begin{align} \label{def:Xk}
 \mathcal{X}_k=\left\{\mathbf{x}(k) \in \mathbb{R}^n: h_s(\xaug (k)) \leq 0, \ s=1,\ldots,m\right\}.
 \end{align}
% where $\xaug (k)$ is the augmented state vector $\xaug (k)=\left[\mathbf{x}^\top(k) \ \ \mathbf{x}^\top(k-1)\right]^\top$ and $h_s(\xaug (k))$ (with $s=1,\ldots,m$) are the polynomial functions in  $\mathbf{x}(k)$ and  $\mathbf{x}(k-1)$ (or equivalently in $\xaug (k)$) defining $\mathcal{X}_k$.
 %\hfill $\blacksquare$
 \end{remark}

 \begin{remark}
 The reformulation of set-membership in the probabilistic framework  is important for two main reasons.
First, it allows us to reinterpret  the operations performed in set-membership estimation and justifies them in terms
  of a probabilistic framework. We have just seen the reinterpretation of prediction and updating in terms of the Chapman-Kolmogorov
  equation and Bayes' rule. We will further investigate this interpretation in the next sections. In particular, in Section \ref{sec:computingmeans}, we will show that the convex membership set computed in set-membership estimation can also be interpreted   as the set of posterior means calculated \DP{with respect to}  the posterior set of probability measures $\mathcal{P}_{\mathcal{X}_k}(\mathbf{X}(k))$   (in the Bayesian setting, we know that the posterior mean is the optimal estimate \DP{with respect to}  a quadratic loss function -- a similar result holds  for the set of posterior means \cite[Sec.5]{benavoli2012g}). This result can also now be applied to set-membership estimation
  because, after this probabilistic interpretation, we are now able to compute expectations. Moreover, in Section \ref{sec:dual} we will also highlight the connection between set-membership
  estimation and the theory of moments (through duality).\\
  Second,  we are now potentially able to combine set-membership and classical probabilistic uncertainty in order to obtain hybrid filters, i.e., stochastic (probabilistic) filters that are for instance able to use information about the bounding region as well as the probabilistic moments (mean and variance) of the noises or that are able to deal with a Gaussian measurement noise and a bounded, with known moments, process noise etc.. A first attempt in this direction is described in \cite{benavoli2013b} for scalar systems. We plan to
  further investigate this direction in future work by using the theory of SOS polynomial optimization, that we also use in the next sections.
 \end{remark}

\section{Computing the support as an inference on the set of probability measures}
\label{sec:computingmeans}
 In the probabilistic formulation of filtering, all the available information at time $k$ is encoded
 in the  posterior probability distribution of the state $\mathbf{x}(k)$ given all the observations $\mathbf{y}^k)$.
 In the set-membership setting, this information is encoded in the updated set of probability measures $\mathcal{P}_{\mathcal{X}_k}(\mathbf{X}(k))$.
 Inferences can then be expressed in terms of expectations computed \DP{with respect to}  this set.
 The set-membership estimation problem  can, for instance, \DP{be reformulated} as follows:
\begin{equation}
\label{eq:prob0}
 \begin{array}{l}
\Omega^*=arg\min\limits_{\Omega \subseteq \mathbb{R}^n} \int\limits_{\Omega} d\mathbf{x}(k)  \\
s.t.\\
 \int\limits_{\Omega} dP(\mathbf{x}(k))= 1, \ \ \forall  P\in \mathcal{P}_{\mathcal{X}_k}(\mathbf{X}(k)).\\
\end{array}
 \end{equation}
 The solution of \eqref{eq:prob0} is the minimum-volume set  $\Omega \subseteq \mathbb{R}^n$, such that
$\Pr(\mathbf{x}(k) \in \Omega)=1$ for all probability measures $\Pr$ in $\mathcal{P}_{\mathcal{X}_k}(\mathbf{X}(k))$ (i.e., with support $\mathcal{X}_k$).\footnote{It is thus the union of all the supports of the probability measures in  $\mathcal{P}_{\mathcal{X}_k}(\mathbf{X}(k))$.} Thus, $\Omega^*$ coincides with $\mathcal{X}_k$.
Since $\mathcal{X}_k$ may be not convex, the problem (\ref{eq:prob0}) is in general difficult to solve.
However, the problem can be simplified  by restricting $\Omega$
to be convex, thus computing  a convex outer-approximation  of $\mathcal{X}_k$.

%To show that,  the following result is used. It can be shown \DP{fornire citazione e/o proporre il risultato come una proposition} that $\mathcal{P}(\mathbf{x})$ is equivalent to the weak$^*$-closure of the convex hull of
%the extreme probability measures:
% \begin{equation}
%\label{eq:probconstroptimequiv}
%\begin{array}{l}
%\mathcal{P}(\mathbf{x})\equiv Co\left\{\delta_{\tilde{x}}(\mathbf{x}):~\tilde{x}\in  \mathcal{X}\right\}.
% \end{array}
%\end{equation}
%By exploiting the result in Eq. \eqref{eq:probconstroptimequiv},
The following theorem shows that computing the minimum-volume convex set $\Omega$  such that $P(\mathbf{x}(k) \in \Omega)=1$ is
equivalent to obtain the set that includes all the possible means computed \DP{with respect to} the probability measure in $\mathcal{P}_{\mathcal{X}_k}(\mathbf{X}(k))$.
\begin{theorem} \label{Th:means}
Assume that $\mathcal{X}_k$ is compact and that $\Omega_1 \subseteq \mathbb{R}^n$ is a convex set defined as follows:
\begin{equation}
\label{eq:prob1}
 \begin{array}{l}
\Omega_1=arg\inf\limits_{\Omega \subseteq \mathbb{R}^n, \Omega~\text{conv. }} \int\limits_{\Omega} d\mathbf{x}(k) \\
s.t.\\
\int\limits_{\Omega} dP(\mathbf{x}(k))=1, \ \ \forall  P \in \mathcal{P}_{\mathcal{X}_k}(\mathbf{X}(k))
\end{array}
 \end{equation}
 Then, it results that $\Omega_1=\mathcal{M}$, with
 \begin{equation}
 \label{eq:prob2}
 \mathcal{M}=\left\{\int\limits_{\mathcal{X}_k}  \mathbf{x}(k) dP(\mathbf{x}(k))\!:\!~~P \in \mathcal{P}_{\mathcal{X}_k}(\mathbf{X}(k))\right\}.
 \end{equation}
 \end{theorem}
 \begin{proof}
  From (\ref{eq:prob1}) it follows  that $\Omega_1$ is the minimum volume
  convex set that includes $\mathcal{X}_k$.
  Thus, if $\mathcal{X}_k$ is convex, then $\Omega_1=\mathcal{X}_k$.
  Hence, from (\ref{eq:probconstroptimequiv}), the equality
  $$
 \int\limits_{\mathcal{X}_k}  \mathbf{x}(k) \delta_{\mathbf{\hat{x}(k)}}(d\mathbf{x}(k)) =\hat{\mathbf{x}}(k),
  $$
 and (\ref{eq:prob2}),  it   immediately follows that $\mathcal{M}=\Omega_1$.
  Conversely assume that $\mathcal{X}_k$ is not convex, then $\Omega_1 \supset \mathcal{X}_k$.
    Since $\Omega_1$ is the minimum volume   convex set that includes $\mathcal{X}_k$, then
    $\Omega_1$ must be equal to the convex-hull of $\mathcal{X}_k$.
  This means that for each $\hat{\mathbf{x}} \in \Omega_1$, there exist
  $\mathbf{z}_1,\mathbf{z}_2 \in \mathcal{X}_k$ such that $w\mathbf{z}_1+(1-w)\mathbf{z}_2=\hat{\mathbf{x}}$
  for some $w \in [0,1]$ (by definition of convex hull).
    Then, consider the probability measure
  \begin{equation}
  w\delta_{\mathbf{z}_1}+(1-w)\delta_{\mathbf{z}_2}.
  \end{equation}
  Because of \eqref{eq:probconstroptimequiv}, it holds:
    \begin{equation}
  w\delta_{\mathbf{z}_1}+(1-w)\delta_{\mathbf{z}_2} \in \mathcal{P}_{\mathcal{X}_k}(\mathbf{X}(k)),
  \end{equation}
  and
   \begin{equation}
 \int\limits_{\mathcal{X}_k}  \mathbf{x}(k) \left(w\delta_{\mathbf{z}_1}(d \mathbf{x}(k))+(1-w)\delta_{\mathbf{z}_2}(d \mathbf{x}(k))\right)=\hat{\mathbf{x}}.
  \end{equation}
  Thus, $\hat{\mathbf{x}}$ belongs to $\mathcal{M}$, and vice versa.
\hfill $\blacksquare$
  \end{proof}

Theorem \ref{Th:means} has the following fundamental implications:
\begin{itemize}
 \item a convex outer-bounding of the set of all the possible means computed \DP{with respect to} the probability measures in $\mathcal{P}_{\mathcal{X}_k}(\mathbf{X}(k))$ (i.e., the set $\mathcal{M}$) is also a convex outer-bounding of the support $\mathcal{X}_k$ of the set of probability measures $\mathcal{P}_{\mathcal{X}_k}(\mathbf{X}(k))$.
 \item the tightest convex outer-bounding of the support $\mathcal{X}_k$ of the set of probability distributions $\mathcal{P}_{\mathcal{X}_k}(\mathbf{X}(k))$ is the set of the  means computed \DP{with respect to} the probability measure in $\mathcal{P}_{\mathcal{X}_k}(\mathbf{X}(k))$.
 \end{itemize}
 We can thus use $\mathcal{M}$ as an outer-approximation of  $\mathcal{X}_k$. Algorithm 1 is therefore
 modified to include the following additional steps.

  \begin{algorithm} \emph{Refinement of Algorithm 1: outer-approximation step}  \\
\begin{description}
   \item[A1.1.3] Outer-approximate $\mathcal{X}_k$ with  $\mathcal{M}$ defined in \eqref{eq:prob2}.
   \item[A1.1.4]  Redefine $\mathcal{P}_{\mathcal{X}_k}(\mathbf{X}(k))\equiv Co\left\{\delta_{\hat{\mathbf{x}}}:~\hat{\mathbf{x}}\in  \mathcal{M}\right\}$.
%    \item[A.2.3] redefine $\mathcal{P}(\mathbf{x}(k)|\mathbf{y}(k))$ as  $\mathcal{P}(\mathbf{x}(k))$ and iterate.
  \end{description}
%\hfill $\blacksquare$
%\hfill $\square$
\end{algorithm}

Unfortunately, Theorem \ref{Th:means} does not provide a constructive way to find the set $\mathcal{M}$. However, by restricting the outer-approximation  of the support $\mathcal{X}_k$ to have a simple form (e.g., a polytope),
Theorem \ref{Th:means} can be still exploited   to determine an outer-bounding set of  $\mathcal{X}_k$. \DP{The following theorem provides results to compute an outer-bounding box of $\mathcal{X}_k$.}

\begin{theorem}[Box approximation]
\label{th:box}
 The minimum volume box that includes  $\mathcal{X}_k$ can be found by solving the following family of
optimization problems
\begin{equation}
\label{eq:box}
 \begin{array}{l}
\underline{\overline{x}}_i^*(k)=\underset{P}{\text{opt}} \int x_i(k) dP(\mathbf{x}(k))  \\
s.t.~  \int\limits_{\mathcal{X}_k} dP(\mathbf{x}(k))= 1.\\
\end{array}
 \end{equation}
 for $i=1,\dots,n$, where by selecting $\text{opt}$ to be $\text{min}$ or $\text{max}$ we obtain
 the half-spaces $\int x_i(k) dP(\mathbf{x}(k))\geq \underline{{x}}_i^*(k)$ and, respectively, $\int x_i(k) dP(\mathbf{x}(k))\leq {\overline{x}}_i^*(k)$ which define the box.
\end{theorem}
The proof of     Theorem \ref{th:box} is provided together with the proof of Theorem \ref{Th:fstar}.
Based on Theorem \ref{th:box}, by computing the lower and upper means of the components $x_1(k),\ldots,x_n(k)$ of the vector $\mathbf{x}(k)$,   the tightest box that outer-approximates $\mathcal{X}_k$ is obtained.
In the following we will discuss how to efficiently solve optimization problems similar to  (\ref{eq:box}) and
how to find an outer-approximation of $\mathcal{X}_k$ that is less conservative than a box.
 For simplicity of notation, in the rest of the paper, the dependence of the state $\mathbf{x}(k)$ and of the set $\mathcal{X}_k$  on the time index $k$ will be dropped, and only used when necessary.

 \section{Exploiting duality}
 \label{sec:dual}
 In this section we discuss how to efficiently solve optimization problems similar to  (\ref{eq:box}).
 In particular, we slightly modify (\ref{eq:box}) in order to be able to determine the more general half-space
  \begin{equation}
\label{eq:hyperspace}
\mathcal{H}=\left\{ \xhyp \in \mathbb{R}^n:   \boldsymbol{\omega}^\top  \xhyp \leq \nu \right\},
 \end{equation}
  where $\boldsymbol{\omega}\in \mathbb{R}^n$, $\nu \in \mathbb{R}$ and $\xhyp= \int    \mathbf{x}  dP(\mathbf{x})$.\footnote{The half-space $\mathcal{H}$ lies on the space of the means.}

\begin{theorem} \label{Th:fstar}
  Let us fix the normal vector $\boldsymbol{\omega}$ defining the half-space $\mathcal{H}$  in \eqref{eq:hyperspace}. Then,   the tightest half-space $\mathcal{H}$ including $\mathcal{M}$ (or equivalently, including $\mathcal{X}$),  is obtained for $\nu= \nu^*$, with
  \begin{equation}
\label{eqn:proofProp2}
 \begin{array}{l}
\nu^*=\max\limits_{P} \int \boldsymbol{\omega}^\top\mathbf{x} dP(\mathbf{x})  \\
s.t.~  \int\limits_{\mathcal{X}} dP(\mathbf{x})= 1.\\
\end{array}
 \end{equation}
\end{theorem}
\begin{proof}
Let $\xhyp = \int    \mathbf{x}  dP(\mathbf{x})$ be a point belonging to $\mathcal{M}$. Let us first \DP{prove} that if $\nu\geq \nu^*$, then $\mathcal{M} \subseteq \mathcal{H}$.
 First, note that:
\begin{align}
\nonumber
\boldsymbol{\omega}^\top \xhyp  \leq \nu^*= & \sup_P  \boldsymbol{\omega}^\top  \int    \mathbf{x}  dP(\mathbf{x})\\
\nonumber
& s.t. \ \int_{\mathcal{X}}dP(\mathbf{x})=1
\end{align}
Therefore, for $\nu\geq \nu^*$,  $\boldsymbol{\omega}^\top  \xhyp \leq \nu^* \leq \nu$, which means that $\xhyp= \int    \mathbf{x}  dP(\mathbf{x})$ also belongs to $\mathcal{H}$ for all $\xhyp \in \mathcal{M}$. Thus,  $\mathcal{H}$ contains  $\mathcal{M}$.
By choosing $\nu=\nu^*$, we obtain the tightest half-space defined by the normal vector $\boldsymbol{\omega}$ that includes $\mathcal{M}$.
%Conversely, if $\Omega_2 \subseteq \mathcal{H}$, then  $\boldsymbol{\omega}^T \int    \mathbf{x}  dP(\mathbf{x}) \leq \nu$ for all $P(\mathbf{x}) \in \mathcal{P}(\mathbf{x})$. Therefore,
%\begin{align}
%\nonumber
%& \sup_P  \boldsymbol{\omega}^T\int    \mathbf{x} dP(\mathbf{x}) \\
%\nonumber
%& s.t. \ \int_{\mathcal{X}}dP(\mathbf{x})=1
%\end{align}
%is also less than or equal to $f$. This completes the proof.
\hfill $\blacksquare$
\end{proof}

% \begin{remark} \label{rem:box}
% An outer-bounding box $\mathcal{B}$ of $\Omega_2$ can be computed trough by setting  $\nu=0$  and $\boldsymbol{\omega}=\mathbf{e}_i$, where $e_i$ is the canonical vector for $\mathbb{R}^n$, with all components equal to $0$, except the $i$-th component which is equal to $1$. \hfill $\blacksquare$
% \end{remark}

It can be observed that (\ref{eqn:proofProp2}) reduces to (\ref{eq:box}) when   $\boldsymbol{\omega}=\mathbf{e}_i$ for $i=1,\dots,n$,
 where $\mathbf{e}_i$ is an element of the natural basis of $\mathbb{R}^n$.
Note that, in Problem \eqref{eqn:proofProp2}: (i) the optimization
variables are the amount of non-negative mass assigned to
each point $\mathbf{x}$ in $\mathcal{X}$ (i.e., the measure $Pr(\mathbf{x})$); (ii) the objective function
and the constraint are linear in the optimization variables. Therefore, \eqref{eqn:proofProp2} is a semi-infinite linear
program (i.e., infinite number of decision variables but finite number of constraints).
By exploiting duality of semi-infinite linear program (see for instance \cite{2001Ala}), we can write the dual of \eqref{eqn:proofProp2}, which is defined as:
\begin{align}
\nonumber
\nu^*=& \inf_{\nu}  \nu\\
\label{eqn:proofProp222}
& s.t. \  \nu \geq \boldsymbol{\omega}^{\top}   \mathbf{x}, \ \ \forall  \mathbf{x} \in \mathcal{X},
\end{align}
which is also a  semi-infinite linear
program (i.e., finite number of decision variables ($\nu$) but infinite number of constraints).
A solution $\nu$ is feasible for Problem  \eqref{eqn:proofProp222} provided that:
$$
 \nu - \boldsymbol{\omega}^{\top}  \mathbf{x} \geq0, ~~\forall  \mathbf{x} \in \mathcal{X}.
$$
%which requires the above polynomial in $\mathbf{x}$ to be non-negative in the semi-algebraic set $\mathcal{X}$.
Hence, checking the feasibility of $\nu$  is equivalent to check the non-negativity of the polynomial  $\nu - \boldsymbol{\omega}^{\top}  \mathbf{x}$ in the set $\mathcal{X}$. % (note that  $\nu - \boldsymbol{\omega}^{\top}  \mathbf{x}$ is a polynomial  in the variable $\mathbf{x}$).
\begin{remark}
The probabilistic formulation of the set-membership estimation described so far is general enough, and it is valid also when the dynamical system in \eqref{eq:setmembersys} is not a polynomial system and when the   uncertainty sets $\mathcal{X}_0,\mathcal{W}_k,\mathcal{V}_k$ in \eqref{eq:sm} are not semialgebraic, but just compact sets. The assumptions of polynomiality  are used in the following to efficiently solve the semi-infinite linear programming problem \eqref{eqn:proofProp222} through convex optimization.
\end{remark}

\subsection{Sum-of-squares polynomials} \label{subsec:SOS}
A sufficient condition for a polynomial to be non-negative over a semialgebraic set
is that it can be written in terms of \emph{sum-of-squares} (SOS) polynomials (see, e.g., \cite{2003Apa}).
\begin{definition}
 A polynomial $\sigma(\xaug)$, with  $\xaug \in \mathbb{R}^{2n}$, of degree $2\relorder$ is a sum-of-squares polynomial, denoted by $\sigma(\xaug) \in \Sigma [\xaug]$, if and only if it can be written as:
   \begin{equation} \label{eqn:SOSdef}
  \sigma(\xaug)=\mathbf{q}_\relorder(\xaug)^\top \mathbf{Q} \mathbf{q}_\relorder(\xaug), %~~\forall ~\mathbf{x} \in \mathbb{R}^n,
 \end{equation}
 where $\mathbf{Q}$  is  a real symmetric positive semidefinite matrix of dimension ${2n+\relorder \choose \relorder}$. The vector of monomials $\mathbf{q}_\relorder(\xaug)$ is defined as in \eqref{eq:moments}. The set of SOS polynomials of degree less then or equal to $2\relorder$ is denoted as $\Sigma_{2\relorder}[\xaug]$.
%
%  $\mathbf{Q}\in \mathbb{R}^{s(d)\times s(d)}$ such that:
%  \begin{equation}
%  \sigma(\mathbf{x})=\mathbf{q}_d(\mathbf{x})^T \mathbf{Q} \mathbf{q}_d(\mathbf{x}), ~~\forall ~\mathbf{x} \in \mathbb{R}^n,
% \end{equation}
% where $\mathbf{q}_d(\mathbf{x})$ is defined in \eqref{eq:moments}.
\end{definition}
Then, for a given integer $\relorder \geq 1$, a sufficient condition for $\nu-\boldsymbol{\omega}^\top   \mathbf{x}$ to be non-negative in $\mathcal{X}$ is (see for instance \cite[Ch. 4]{lasserre2010moments}):
  \begin{equation} \label{opt:algappSDPv3}
   \begin{array}{l}
\nu - \boldsymbol{\omega}^\top   \mathbf{x}   =\sigma_0(\xaug)-\sum\limits_{s=1}^m \sigma_s(\xaug) h_s(\xaug)~~\forall~\xaug \in \mathbb{R}^{2n}\\
   \sigma_0(\xaug), \sigma_1(\xaug), \ldots, \sigma_m(\xaug) \in \Sigma_{2\relorder}[\xaug], %\\
  % deg(\!\sigma_0\!), deg(\!\sigma_1h_1(\mathbf{x})\!), \ldots, deg(\!\sigma_mh_m(\mathbf{x})\!) \leq 2 \delta,\\
   \end{array}
    \end{equation}
  where $h_s(\xaug)$ (with $s=1,\ldots,m$) are the polynomial nonpositive inequality constraints defining the semialgebraic set $\mathcal{X}$. In order to avoid confusion, we would like to stress that also $\nu - \boldsymbol{\omega}^\top  \mathbf{x}$ is a polynomial in the variable $\xaug$. \DP{In fact,} we remind that the augmented state  $\xaug(k)$ is defined as: $\xaug(k)=\left[\mathbf{x}^\top(k) \ \ \mathbf{x}^\top(k-1)\right]^\top$.

   % in \eqref{def:Xk}. % and  $\delta \in \mathbb{N}$ is a given integer such that $2\delta \geq \max\{deg(h_1(\mathbf{x})), \ldots, deg(h_m(\mathbf{x}))\}$.
The following (more conservative) optimization problem can be then solved instead of  \eqref{eqn:proofProp222}:
  \begin{equation} \label{opt:algappSDPv4}
  \begin{split}
\nu^{**}&= \inf_{\nu,\sigma_s}  \nu\\
   & \nu - \boldsymbol{\omega}^\top  \mathbf{x}   = \sigma_0(\xaug)-\sum_{s=1}^m \sigma_s(\xaug) h_s(\xaug), ~~\forall~\xaug \in \mathbb{R}^{2n}\\
  & \sigma_0(\xaug), \sigma_1(\xaug), \ldots, \sigma_m(\xaug) \in \Sigma_{2\relorder}[\xaug].
  % & deg(\!\sigma_0\!), deg(\!\sigma_1h_1(\mathbf{x})\!), \ldots, deg(\!\sigma_mh_m(\mathbf{x})\!) \leq 2 \delta,\\
  \end{split}
  \end{equation}
Note that, by rewriting the SOS polynomials $\sigma_s(\xaug)$ (with $s=0,\ldots,m$) as in \eqref{eqn:SOSdef}, Problem \eqref{opt:algappSDPv4} can be also rewritten as:
  \begin{equation} \label{opt:algappSDP}
  \begin{split}
\nu^{**}=& \inf_{\nu,\mathbf{Q}_s}  \nu\\
    \nu - \boldsymbol{\omega}^\top  \mathbf{x}   = & \mathbf{q}_\relorder(\xaug)^\top \mathbf{Q}_0 \mathbf{q}_\relorder(\xaug)+\\
   -& \sum_{s=1}^m \mathbf{q}_\relorder(\xaug)^\top \mathbf{Q}_s \mathbf{q}_\relorder(\xaug) h_s(\xaug), ~~\forall~\xaug \in \mathbb{R}^{2n}\\
  &\mathbf{Q}_s \succeq 0, \ \ s=0,\ldots,m.
  % & deg(\!\sigma_0\!), deg(\!\sigma_1h_1(\mathbf{x})\!), \ldots, deg(\!\sigma_mh_m(\mathbf{x})\!) \leq 2 \delta,\\
  \end{split}
  \end{equation}

Some remarks:
\begin{enumerate}
 \item Problem \eqref{opt:algappSDP} is a \emph{semidefinite programming} (SDP) problem \cite{2003Apa,2003AChGaTe}, thus convex. In fact, checking if  the polynomial $\nu-\boldsymbol{\omega}^T\mathbf{x}$ is equal to $\mathbf{q}_\relorder(\xaug)^\top \mathbf{Q}_0 \mathbf{q}_\relorder(\xaug)-\sum_{s=1}^m \mathbf{q}_\relorder(\xaug)^\top \mathbf{Q}_s \mathbf{q}_\relorder(\xaug) h_s(\xaug)$ for all $\xaug \in \mathbb{R}^{2n}$ leads to linear equalities in  $\nu$ and in the matrix coefficients $\mathbf{Q}_s$ (with $s=1,\ldots,m$).  Besides, enforcing  $\sigma_0(\xaug), \sigma_1(\xaug), \ldots, \sigma_m(\xaug)$ to be sum of square polynomials leads to \emph{linear matrix inequality} (LMI) constraints in the coefficients of $\sigma_0(\xaug), \sigma_1(\xaug), \ldots, \sigma_m(\xaug)$ (i.e., $\mathbf{Q}_s \succeq 0$).  %Therefore,  Problem \eqref{opt:algappSDP} is a semidefinite programming (SDP) problem \cite{2003AChGaTe,2003Apa}.
     \item For $\nu=\nu^{**}$,   the robust constraint $  \nu^{**} - \boldsymbol{\omega}^{\top}   \mathbf{x} \geq 0 \ \ \forall  \mathbf{x} \in \mathcal{X}$ appearing in Problem \eqref{eqn:proofProp222} is guaranteed to be satisfied. As matter of fact, for all $\xaug \in \tilde{\mathcal{X}}$, $h_s(\xaug) \leq 0$ (with $s=1,\ldots,m$) by definition of $\tilde{\mathcal{X}}$. Furthermore,  the SOS polynomials  $\sigma_s(\xaug)=\mathbf{q}_\relorder(\xaug)^\top \mathbf{Q}_s \mathbf{q}_\relorder(\xaug)$ (with $s=0,\ldots,m$) are always nonnegative over $\mathbb{R}^{2n}$ as $\mathbf{Q}_s\succeq 0$. Thus, both the left and the right side of the equation in Problem \eqref{opt:algappSDP} are nonnegative  for all $\xaug \in \tilde{\mathcal{X}}$.
%      Since $\mathcal{X}$ is the projection in $\mathbf{x}$ of $\tilde{\mathcal{X}}$, it follows that $ \nu^{**} - \boldsymbol{\omega}^{\top}   \mathbf{x} \geq 0 \ \ \forall  \mathbf{x} \in \mathcal{X}$.
  \item Since the equality constraint in \eqref{opt:algappSDP} gives only a sufficient condition for the non-negativity of $\nu - \boldsymbol{\omega}^\top   \mathbf{x}$
  on $\mathcal{X}$, it follows that   $\nu^{*}\leq \nu^{**}$. Therefore,   conservativeness is introduced in solving \eqref{opt:algappSDP} instead of \eqref{eqn:proofProp222}, as highlighted in Corollary \ref{co:1}.
  \item However, according to the \emph{Putinar's Positivstellensatz} \DP{(see, e.g., \cite{Putinar93} and \cite[Ch. 3]{2009Ala}),} a  polynomial  which is nonnegative over a compact semialgebraic set $\mathcal{X}$ can exactly always be written as a combination of SOS polynomials, provided that the degree of the SOS polynomials  $\sigma_0(\xaug),\ldots,\sigma_m(\xaug)$ is large enough.
  In other words, we can make $\nu^{**}$ close to $\nu^{*}$ by increasing the degree of the SOS.
However, in practice it often happens that the relaxed solution $\nu^{**}$ and the optimal one $\nu^{*}$  coincide with each other for small values of the SOS degree $2\relorder$. %\hfill $\blacksquare$\\
\end{enumerate}

\begin{corollary}
\label{co:1}
The set $\mathcal{M}$ is guaranteed to belong to the halfspace $\mathcal{H}:\boldsymbol{\omega}^\top \mathbf{x} \leq \nu^{**}$, i.e.
 \begin{equation}
 \mathcal{M} \subseteq \mathcal{H}.\\
 \end{equation}
\end{corollary}

\begin{proof}
The proof straightforwardly follows from  Theorem \ref{Th:fstar} and $\nu^* \leq  \nu^{**}$. \hfill $\blacksquare$
\end{proof}

\begin{example}
 \label{ex:1}
Let us consider the discrete-time polynomial system described by the difference equations:
\begin{equation}
\label{eq:ex1sys}
\begin{array}[pos]{rcl}
   x_1(k)&\!\!=\!\!& x_1\!(k\!-\!1)x_2(k\!-\!1)\!(x_1\!(k\!-\!1)+x_2\!(k\!-\!1)) \!+\!w_1(k\!-\!1),\\
   x_2(k)&\!\!=\!\!& x_1\!(k\!-\!1)x_2(k\!-\!1)\!(2x_1\!(k\!-\!1)+x_2\!(k\!-\!1))\! +\! w_2(k\!-\!1).\\
\end{array}
\end{equation}
%with initial conditions $\mathbf{x}(0)=[x_1(0) \ \ x_2(0)]^\top=[0.1 \ \ -0.1]^\top$.
The output equation is given by:
$
\mathbf{y}(k)=x_1(k)+x_2(k)+\mathbf{v}(k)$.
%The goal is to compute (an outer approximation of) the state uncertainty set $\mathcal{X}_k$ containing the state $\mathbf{x}(k)$ at time $k=1$, knowing that:
The following conditions are assumed: (i) the initial state $\mathbf{x}(0)$ belongs to $\mathcal{X}_0=\left\{\mathbf{x}(0):~\|\mathbf{x}(0)\|_{2} \leq  0.2\right\}$, the process noise  $\mathbf{w}(k)=[w_1(k) \ \ w_2(k)]^\top$ is bounded by $\|\mathbf{w}(k)\|_{2} \leq 0.4$, and
the measurement noise by $\|\mathbf{v}(k)\|_{\infty} \leq 0.5$. The observed output $\mathbf{y}(k)$ at time $k=1$ is $\mathbf{y}(k)=0$.
 We are interested in computing an half-space $\mathcal{H}:\boldsymbol{\omega}^\top \xhyp \leq \nu$ containing the state uncertainty set $\mathcal{X}_k$ (or equivalently $\mathcal{M}$) at time $k=1$. The normal vector $\boldsymbol{\omega}$ characterizing $\mathcal{H}$ is fixed and it is equal to $\boldsymbol{\omega}=[-1 \ \ -0.5]^\top$.  In order to compute the constant parameter $\nu$ defining $\mathcal{H}$, the SDP   Problem \eqref{opt:algappSDP}
 with $\xaug(1)=[\mathbf{x}^T(1) \ \ \mathbf{x}^T(0)]^T$ and
 \begin{align}
 \label{eq:defhs}
 h_1(\xaug(1))\!:&  x_1(0)^2 + x_2(0)^2 - 0.2^2 \leq 0,\\
 h_2(\xaug(1))\!:&  \underset{w^2_1(0)}{\underbrace{\left(x_1(1)\!-\! x_1\!(0)x_2(0)\!(x_1\!(0)+x_2\!(0))\right)^2}}\!+  \nonumber \\
                 &  \underset{w^2_2(0)}{\underbrace{\left(x_1(1)\!-\! x_1\!(0)x_2(0)\!(2x_1\!(0)+x_2\!(0))\right)^2}} -0.4^2 \!\leq\! 0,  \nonumber\\
 h_3(\xaug(1))\!:&  \underset{\mathbf{v}(1)}{\underbrace{y(1)-x_1(1)-x_2(1)}}-0.5 \leq 0,  \nonumber \\
 h_4(\xaug(1))\!:&  -\underset{\mathbf{v}(1)}{\left(\underbrace{y(1)-x_1(1)-x_2(1)}\right)}-0.5 \leq 0,  \nonumber \\
   \end{align}
    is solved  for a SOS degree $2\relorder=4$. The \emph{SOStools} \cite{sostools} has been used to easily handle the SOS polynomials appearing in \eqref{opt:algappSDP}. The   CPU time taken by the solver \emph{SeDuMi} \cite{1999Ast} to compute a solution of the SDP   Problem \eqref{opt:algappSDP} on a 2.40-GHz Intel Pentium IV  with 3 GB of RAM is 2.1 seconds.  The computed half-space  $\mathcal{H}$ is plotted in Fig. \ref{fig:trueX}, along with the   true  state uncertainty set $\mathcal{X}_1$. According to Theorem \ref{Th:fstar} and
Corollary \ref{co:1}, $\mathcal{X}_1$ is included in the half-space $\mathcal{H}$. Note also that, although the original robust optimization problem \eqref{eqn:proofProp222} has been replaced with the SDP problem \eqref{opt:algappSDP}, the computed parameter $\nu^{**}$ defining $\mathcal{H}$ is such that the hyperplane $\boldsymbol{\omega}^\top  \mathbf{x} = \nu^{**}$ is ``almost'' tangent to the set $\mathcal{X}_1$. Thus, only a small level of conservativeness is introduced in using SOS.

\begin{figure}
\includegraphics[scale=0.5]{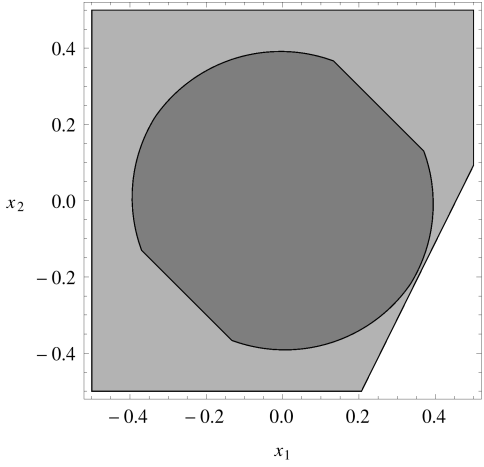}
 \caption{True state uncertainty set $\mathcal{X}_1$ (dark grey region) and half-space $\mathcal{H}: -\xhyp_1 -
  0.5 \xhyp_2 \leq  0.45$ (light gray region).} \label{fig:trueX}
\end{figure}

\end{example}

\section{Computation of the minimum-volume polytope containing $\mathcal{M}$}
In the previous section, given the normal vector  $\boldsymbol{\omega}$ defining
the half-space $\mathcal{H}$ in \eqref{eq:hyperspace}, we have shown how to compute, through convex optimization, the constant parameter $\nu$ such that $\mathcal{M} \subset \mathcal{H}$.

Now consider the following family of  half-spaces:
$$
\mathcal{H}_j=\left\{ \xhyp \in \mathbb{R}^n:   \boldsymbol{\omega}_j^\top \xhyp \leq \nu_j \right\},
 $$
for $j=1,\dots,J$ with $J\geq n+1$.
Our goal is to choose the normal vectors $\boldsymbol{\omega}_j$, along with the constant parameters $\nu_j$,  defining the half-spaces $\mathcal{H}_j$ such that
\begin{enumerate}
 \item $ \mathcal{M} \subseteq \mathcal{S}=\bigcap_{j=1}^J \mathcal{H}_j$;
 \item the polytope $\mathcal{S}$ has minimum volume.
\end{enumerate}
In other words, now also the normal vectors $\boldsymbol{\omega}_j$ for $j=1,\dots,J$  have to be optimized.
%First of all observe that, being $\nu_j,f_j$ optimization variables,  without loss of generality we can include $f_j$ in $\nu_j$ and consider the hyperspaces:
%$$
%\mathcal{H}_j=\left\{ {\boldsymbol{\omega}_j^T \int \mathbf{x} dP(\mathbf{x})-\nu_j} \leq 0:~~\forall ~P(\mathbf{x})\in \mathcal{P}(\mathbf{x})\right\}.
% $$
Then, we can formulate the problem we aim to solve as:
\begin{equation} \label{eqn:problinhull}
\begin{split}
 %& \mathcal{S}^{*}=\arg
 &\inf_{\mathcal{S}}\int_\mathcal{S} d\mathbf{x} ~~ \mbox{s.t. }  \mathcal{M} \subseteq \mathcal{S},
\end{split}
\end{equation}
where $\mathcal{S}$ in \eqref{eqn:problinhull} is constrained to be a polytope. % defined by (at most) $J$ half-spaces.
There are two  main aspects making \eqref{eqn:problinhull} a challenging problem, i.e.,
\begin{enumerate}
  \item the  minimum-volume polytope outer-approximating  a generic compact set in $\mathbb{R}^n$ might not exist. For instance, if $\mathcal{M}$ is an ellipsoid, its  convex hull is described by an infinite number of half-spaces, namely all the supporting hyperplanes at every boundary point of  $\mathcal{M}$.
       \item the problem of computing the exact volume $\int_\mathcal{S} d\mathbf{x}$ of a polytope $\mathcal{S}$ in $\mathbb{R}^n$ is $\#P$-hard (see, e.g. \cite{dyer1988complexity,bueler2000exact}. \DP{The interested reader is also referred to \cite{arora2009computational} for  details on $\#P$-hard problems).}   Although several algorithms have been proposed in the literature to compute the volume of a polytope  $\mathcal{S}$ through triangulation  \cite{cohen1979two,allgower1986computing,lasserre1983analytical,bemporad2004inner}, Gram's relation \cite{lawrence1991polytope}, Laplace transform \cite{lasserre2001laplace} or randomized methods \cite{smith1984efficient,dyer1991random,wiback2004monte}, all the approaches mentioned above require an exact description of the polytope $\mathcal{S}$ in terms of its half-space or vertex representation. However, in our case, the parameters $\boldsymbol{\omega}_j,\nu_j$ defining the half-spaces $\mathcal{H}_j$ are unknown, as determining $\boldsymbol{\omega}_j,\nu_j$ is part of  the problem itself.
\end{enumerate}
 In the following paragraph we present a greedy algorithm   to evaluate an approximation of the minimum-volume polytope   outer-approximating the set $\mathcal{M}$. %The proposed algorithm makes use of results from real algebraic geometry on the representation of positive polynomials  as \emph{sum-of-squares} (SOS) polynomials.
 %The paper is organized as follows.  A general overview  of the algorithm is given in Section \ref{Sec:main}, while mathematical details are reported in Section \ref{Sec:techdet}. A refinement of the main algorithm presented in Section \ref{Sec:main} is proposed in Section \ref{Sec:extension}. Finally, an illustrative example is reported in Section \ref{Sec:example}.
%

% \subsection{Polytopic outer approximation: main algorithm}
% In this paragraph, the main ideas underlying the   computation of a polytopic outer approximation  of the  set $\Omega_2$ are given. Technical details  are reported in Section \ref{}.

\subsection{Approximation of the objective function} \label{susec:area}
As already pointed out in the previous paragraph, one of the main problems in solving \eqref{eqn:problinhull} is that an analytical expression for the computation of  the volume of a polytope $\mathcal{S}$ in $\mathbb{R}^n$ is not available   and the polytope $\mathcal{S}$ is unknown, as  computing $\mathcal{S}$ is part of  the problem itself.
In order to overcome such a problem, a Monte Carlo integration approach \cite{robert2004monte} is used here to approximate the volume of   $\mathcal{S}$. Specifically, given an outer-bounding box $\mathcal{B}$ of the set $\mathcal{M}$ (which can be computed as discussed in Theorem \ref{th:box}) and a sequence of $N$ random points $\{p_i\}_{i=1}^N$ \DP{independent and} uniformly distributed in $\mathcal{B}$, the integral  $\int_\mathcal{S} d\mathbf{x}$ can be approximated as:
\begin{equation} \label{eqn:app}
\int_\mathcal{S} d\mathbf{x} \approx Vol(\mathcal{B})\frac{1}{N}\sum_{i=1}^N I_{\{\mathcal{S}\}}(p_i),
\end{equation}
where $Vol(\mathcal{B})$ is the volume of the box  $\mathcal{B}$ and $I_{\{\mathcal{S}\}}(p_i)$ is the indicator function of the (unknown) polytope $\mathcal{S}$ defined as
\begin{equation}
I_{\{\mathcal{S}\}}(p_i) = \left\{\begin{array}{lll}
                                    1 & \mbox{ if }  p_i \in  \mathcal{S}\\
                                    0 & \mbox{ otherwise }
                                  \end{array}
\right.
\end{equation}

\begin{remark}
It is worth remarking that:
$$\mathbb{E}\left[Vol(\mathcal{B})\frac{1}{N}\sum_{i=1}^N I_{\{\mathcal{S}\}}(p_i)\right]=Vol(\mathcal{S}),$$
\DP{where the expectation is taken \DP{with respect to} the random variable $p_i$.
 Furthermore,} because of the strong law of large numbers,
\begin{equation}
\lim_{N \rightarrow \infty}Vol(\mathcal{B})\frac{1}{N}\sum_{i=1}^N I_{\{\mathcal{S}\}}(p_i)=Vol(\mathcal{S}) \ \ \mathrm{w.p.}\  1,
\end{equation}
\DP{where $\mathrm{\emph{w.p.}}\ 1$ is for \emph{with probability} $1$.} {For finite samples $N$, the level of  accuracy 
of the approximation in \eqref{eqn:app} depends on the shape of the set $\mathcal{S}$ as well as on the volume of the 
outer box $\mathcal{B}$.}
The reader is referred to \DP{as} \cite{robert2004monte} for details on Monte Carlo integration methods.
% $\blacksquare$
\end{remark}

On the basis of \eqref{eqn:app}, the volume minimization of problem \eqref{eqn:problinhull} can be then approximated as
\begin{equation} \label{appr:problinhullappr}
\min_{\mathcal{S} \in \mathbf{S}}\sum_{i=1}^N I_{\{\mathcal{S}\}}(p_i) ~~~ \mbox{s.t. }  \mathcal{M} \subseteq 
\mathcal{S}
\end{equation}

%\begin{remark} \label{rem:outb}
%An outer-bounding box $\mathcal{B}$ of the semialgebraic set $\mathcal{S}$ can be evaluated by computing the minimum and maximum value of each component of the vector $x$ over the semialgebraic set $\mathcal{S}$, that is by solving the polynomial optimization problems
%\begin{subequations} \label{eq:outb}
%\begin{equation}
%\underline{x}^{(k)}=\min_{x \in \mathbb{R}^n} x^{(k)}\ \ s.t. \ \ x \in \mathcal{S}, \; \; k=1,\ldots,n;
%\end{equation}
%\begin{equation}
%\overline{x}^{(k)}=\max_{x \in \mathbb{R}^n} x^{(k)}\ \ s.t. \ \ x \in \mathcal{S}, \; \; k=1,\ldots,n,
%\end{equation}
%\end{subequations}
%where $x^{(k)}$ denotes the $k$-th component of the vector $x$. A lower and an upper bound of $\underline{x}^{(k)}$ and $\overline{x}^{(k)}$, respectively, can be then computed by exploiting the techniques  presented in \cite{2001Ala,2003AChGaTe,2003Apa,lasserre2010moments} to relax the polynomial optimization problems \ref{eq:outb} into a sequence of \emph{semidefinite programming }(SDP) problems. \hfill $\square$\\
%\end{remark}

In the following subsection, we describe a greedy procedure aiming at computing an approximation of the  minimization problem \eqref{appr:problinhullappr}.

\subsection{A greedy approach for solving \eqref{appr:problinhullappr}}
The key steps of the approach proposed in this section  to compute a polytopic outer-approximation $\mathcal{S}$ of the   set $\mathcal{M}$ are  summarized in Algorithm 2. %We remind that the number of half-spaces defining the polytope $\mathcal{S}$ can be arbitrarly is a-priori specifiecied limited by \\

\begin{algorithm} \emph{Algorithm 2: Polytopic outer approximation $\mathcal{S}$ of $\mathcal{M}$}  \\
 \noindent [\emph{input} \!] List $\mathcal{L}=\{p_i\}_{i=1}^N$ of $N$ random points uniformly distributed in the box $\mathcal{B}$.
\begin{description}
  %\item[B.1] Generate a list $\mathcal{L}=\{p_i\}_{i=1}^N$ of $N$ random points uniformly distributed in $\mathcal{B}$.
  \item[A2.1] Set $j=1$.
  \item[A2.2] Compute the half-space $\mathcal{H}_j$, defined as $\mathcal{H}_j:\boldsymbol{\omega}_j^\top \xhyp-\nu_j \leq 0$ (with $\boldsymbol{\omega}_j\neq 0$), that contains the minimum number of points in the list $\mathcal{L}$ and such that $\mathcal{M}$ is included in $\mathcal{H}_j$, i.e.,
  \begin{equation} \label{opt:algv4}
  \begin{split}
\boldsymbol{\omega}^*_j,\nu^*_j= & \mbox{arg} \min_{\substack{     \boldsymbol{\omega}_j \in \mathbb{R}^n\\ \nu_j \in \mathbb{R} }}
\sum_{i=1}^N I_{\{\mathcal{H}_j\}}(p_i) \\
   & \mbox{s.t. } \\
   & \boldsymbol{\omega}_j \neq 0 \\
   & \mathcal{M} \subseteq \mathcal{H}_j \\
  % & \omega^{{\tr}}x+b^{} \geq 0 \ \ \ \forall x \in \mathcal{S} \\
   & p_i \in \mathcal{L}, \ \ \ i=1,\ldots,N
  \end{split}
  \end{equation}
\item[A2.3] Collect all the points $p_i \in \mathcal{L}$  belonging  to the half-space  $\mathcal{H}_j$ (computed through \eqref{opt:algv4}) in a list $\mathcal{L}_j$. Let $N_j$ be the number of elements of $\mathcal{L}_j$.
\item[A2.4] If $N_j < N$, then   $\mathcal{L} \leftarrow \mathcal{L}_j$, $N \leftarrow N_j$, $j \leftarrow j+1$ and go to step A2.2. Otherwise,  set $J=j-1$ and go to step  A2.5.
\item[A2.5] Define the polytope $\mathcal{S}$ as
$
\mathcal{S}=\mathcal{B} \cap \bigcap_{j=1}^J\mathcal{H}_j.
$
\end{description}
\noindent [\emph{output} \!] Polytope $\mathcal{S}$.
%\hfill $\blacksquare$
%\hfill $\square$
\end{algorithm}

Algorithm 2 generates a sequence of half-spaces $\mathcal{H}_1,\ldots, \mathcal{H}_J$ as follows. First, the half-space $\mathcal{H}_1$ that minimizes an approximation of the \DP{volume} of the polytope $\mathcal{B} \cap \mathcal{H}_1$ is computed. The approximation is due to the fact that the \DP{volume} of $\mathcal{B} \cap \mathcal{H}_1$, given by the integral $  \int_{\mathcal{B} \cap \mathcal{H}_1} d\mathbf{x} $, is approximated (up to the constant $\tfrac{Vol(\mathcal{B})}{N}$) by $ \sum_{i=1}^N I_{\{\mathcal{H}_1\}}(p_i)$ (corresponding to the objective function of problem \eqref{opt:algv4}). Then, the new half-space $\mathcal{H}_2$ that minimizes an approximation of the \DP{volume} of the polytope $\mathcal{B} \cap \mathcal{H}_1  \cap \mathcal{H}_2$ is generated. In order to approximate the \DP{volume} of $\mathcal{B} \cap \mathcal{H}_1  \cap \mathcal{H}_2$, all the points $p_i$ of the list $\mathcal{L}=\{p_i\}_{i=1}^N$ that do not belong to the polytope $\mathcal{B} \cap \mathcal{H}_1$ are discarded, and all and only the points belonging to $\mathcal{B} \cap \mathcal{H}_1$ are collected in a new list $\mathcal{L}_1=\{p_i\}_{i=1}^{N_1}$  (step A2.3).  The \DP{volume} of $\mathcal{B} \cap \mathcal{H}_1  \cap \mathcal{H}_2$  is then approximated by $  \sum_{i=1}^{N_1} I_{\{\mathcal{H}_2\}}(p_i)$, with $p_i \in \mathcal{L}_1$. The procedure is repeated until $N_{J+1}=N_{J}$ (step A2.4), which means that the number of samples $p_i$ belonging to the polytope $\mathcal{B} \cap \mathcal{H}_1 \cap \ldots \cap \mathcal{H}_{J+1}$ is equal to the number of samples $p_i$ belonging to the polytope $\mathcal{B} \cap \mathcal{H}_1 \cap \ldots \cap \mathcal{H}_{J}$. Note that, because of the constraint $\mathcal{M} \subseteq \mathcal{H}_j$ appearing in optimization problem \eqref{opt:alg}, the half-spaces $\mathcal{H}_1,\ldots, \mathcal{H}_J$ are guaranteed to contain the  set $\mathcal{M}$, and thus $ \mathcal{S}=\mathcal{B} \cap \bigcap_{j=1}^J\mathcal{H}_j$ is an outer approximation of $\mathcal{M}$.
Finally, we would like to remark that, in case we are interested also in bounding the maximum number of half-spaces 
defining the  polytopic outer approximation $\mathcal{S}^{}$, Algorithm 2 can be stopped after an a-priori specified 
number of iterations.

\begin{example}
 \label{ex:2}
 Let us consider again Example \ref{ex:1}.
The first steps of Algorithm 2  are visualized in Fig. \ref{fig:Alg1}.
An outer-bounding box $\mathcal{B}$ of the true state uncertainty set (dark gray region) is first computed (Fig. (a)).
A set of $80$ random points (black dots) uniformly distributed in  $\mathcal{B}$  is generated (Fig. (b)). The half-space $\mathcal{H}_1$ containing the true state uncertainty set and the minimum number of points is computed. The points which do not belong to $\mathcal{H}_1$ are discarded (gray dots in Fig. (c)).
A new half-space $\mathcal{H}_2$ containing the true state uncertainty set and the minimum number of black dots is computed (Fig. (d)). Again, the points that do not belong to $\mathcal{H}_1 \cap \mathcal{H}_2$ are discarded (gray dots in Fig. (d)). The procedure terminates when no more black points can be discarded.

\begin{figure}[h]
        \centering
         \begin{tabular}{cc}
                     \includegraphics[width=0.22\textwidth]{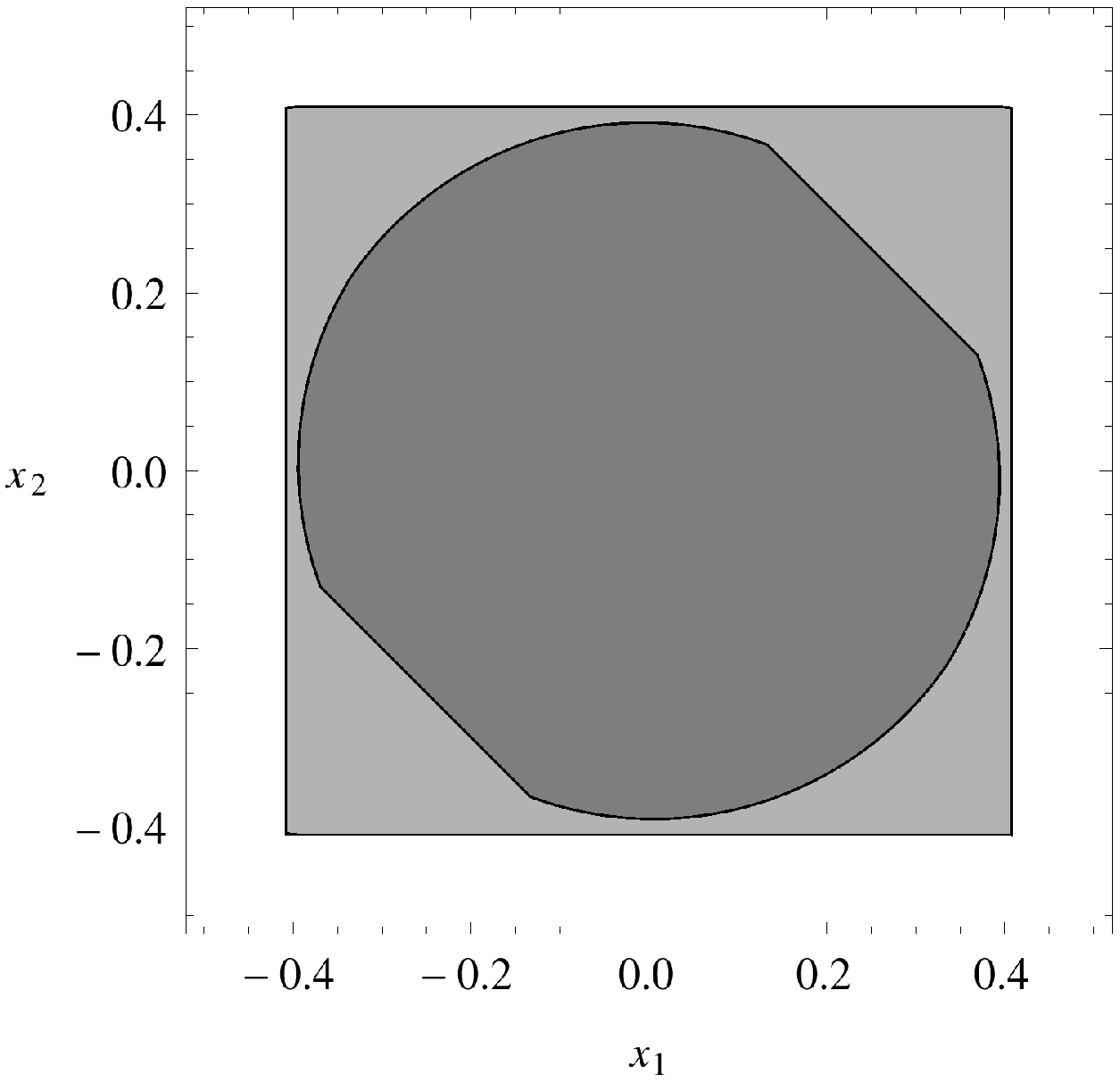} &  \includegraphics[width=0.22\textwidth]{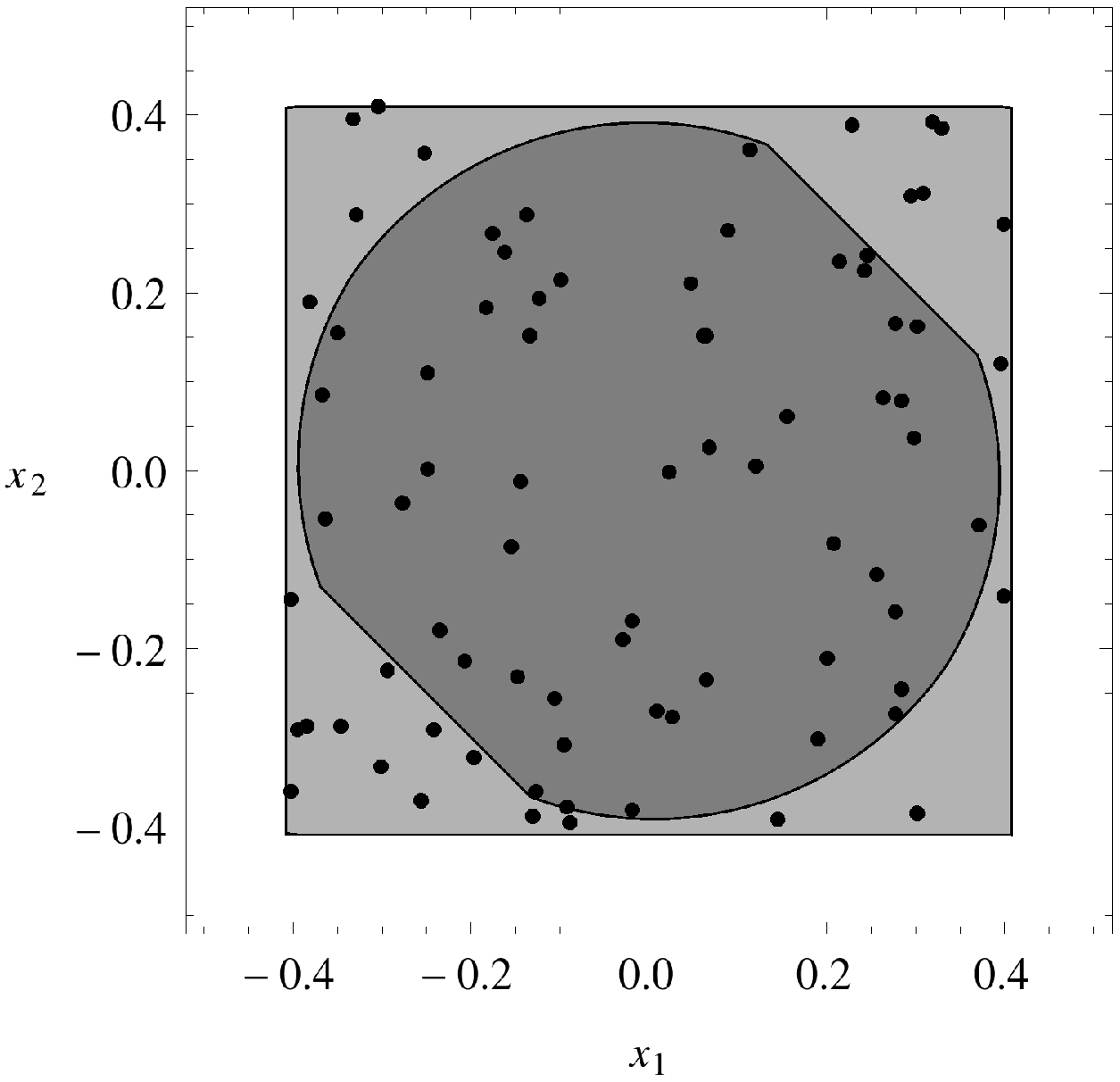} \vspace{-0.2cm}\\
                 \ \ \ \  \ \    (a)  &  \ \ \ (b) \vspace{0.2cm}\\
                      \includegraphics[width=0.22\textwidth]{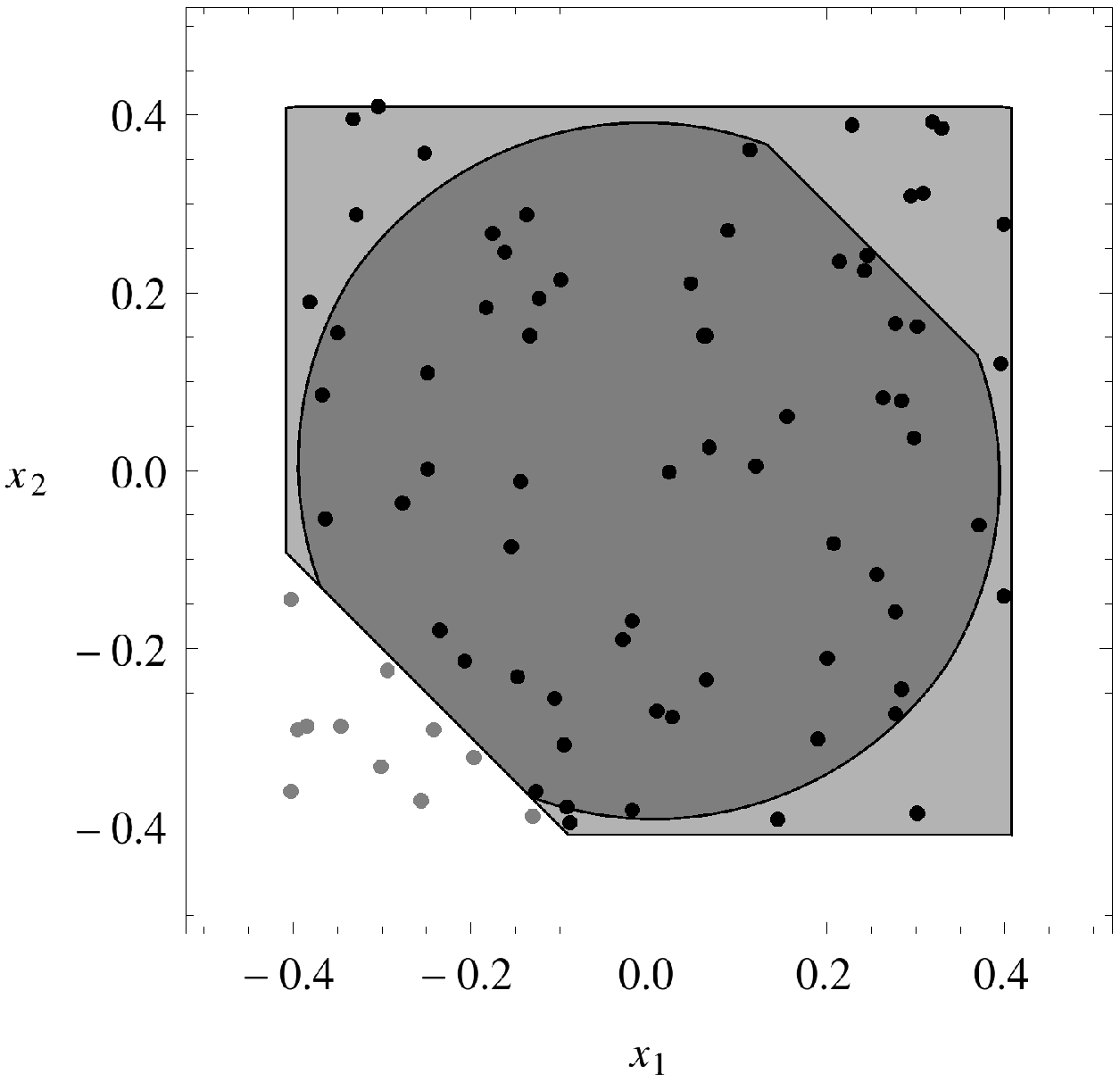} & \includegraphics[width=0.22\textwidth]{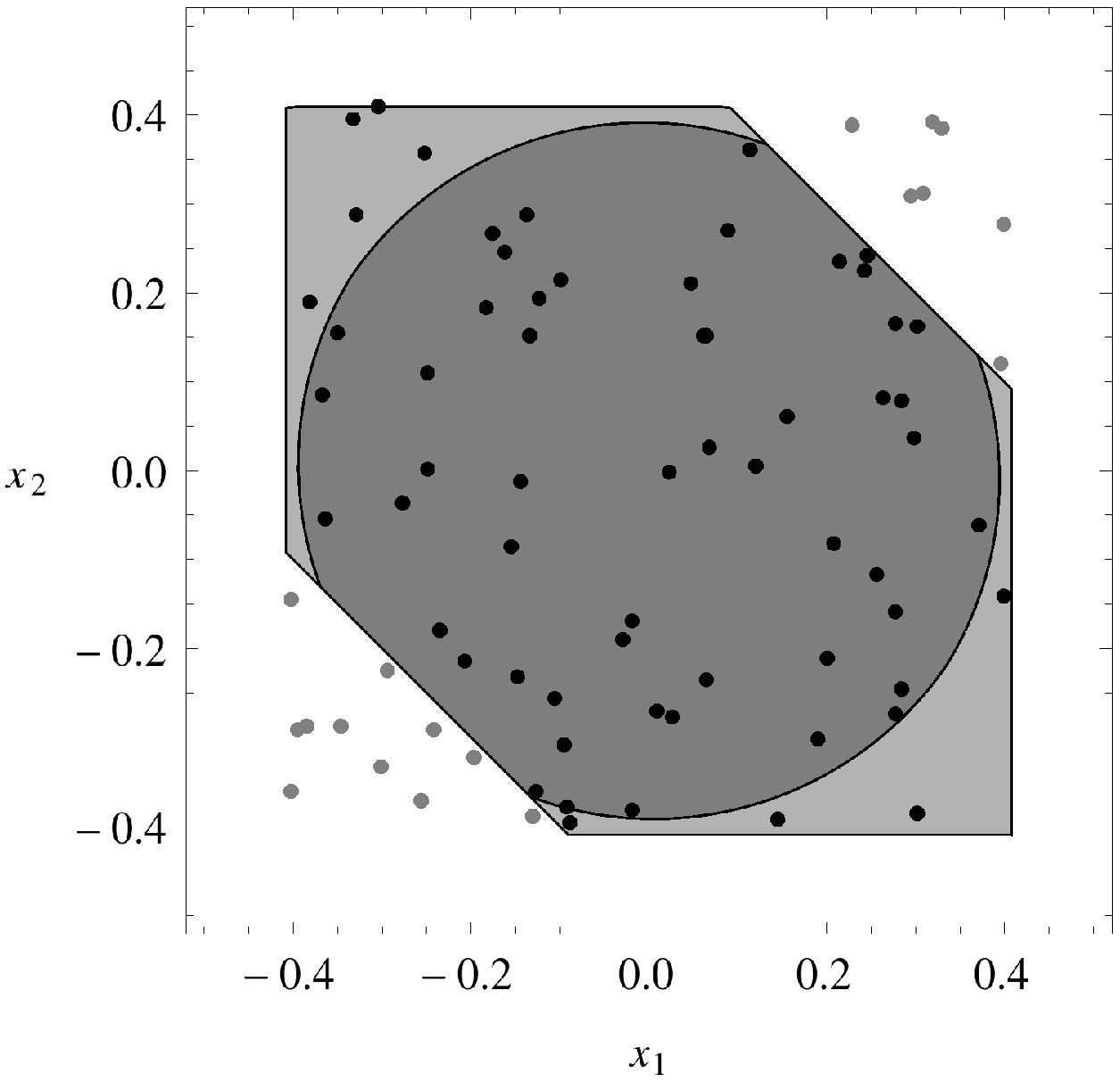} \vspace{-0.2cm}\\
                \ \ \ \    \ \    (c)  &  \ \ \ (d)
         \end{tabular}
        \caption{First steps of Algorithm 2.}\label{fig:Alg1}
\end{figure}

\end{example}

Technical details of step \DP{A2.2,} which is the core of Algorithm 2, are provided in the following sections. %In particular, we will show how to handle the constraint $\Omega_2 \subseteq \mathcal{H}_j$ and how relaxing the nonconvex objective function $\sum_{i=1}^N I_{\{\mathcal{H}_j\}}(p_i)$. \\

\subsection{Approximation of the indicator functions}
Note that the objective function of problem \eqref{opt:algv4} is noncontinuous and nonconvex since it is the sum of the indicator functions $I_{\{\mathcal{H}_j\}}(p_i)$ defined as
\begin{equation}
I_{\{\mathcal{H}_j\}}(p_i) = \left\{\begin{array}{lll}
                                    1 & \mbox{ if } \boldsymbol{\omega}_j^\top p_i-\nu_j \leq 0,\\
                                    0 & \mbox{ if } \boldsymbol{\omega}_j^\top p_i-\nu_j > 0.
                                  \end{array}
\right.
\end{equation}
We then transform it in a convex objective function. Each indicator function $I_{\{\mathcal{H}_j\}}(p_i)$ is here approximated by the convex function $R_{\{\mathcal{H}_j\}}(x_i)$ defined as
\begin{equation}
R_{\{\mathcal{H}_j\}}(p_i) = \left\{\begin{array}{lll}
                                    -\boldsymbol{\omega}_j^\top p_i+\nu_j^{} & \mbox{ if }\boldsymbol{\omega}_j^\top p_i-\nu_j^{} \leq 0,\\
                                    0 & \mbox{ if } \boldsymbol{\omega}_j^\top p_i-\nu_j^{} > 0.
                                  \end{array}
\right.
\end{equation}
A plot of the functions $I_{\{\mathcal{H}_j\}}(p_i)$ and $R_{\{\mathcal{H}_j\}}(p_i)$ is given in Fig. \ref{fig:indfuncapp1}.
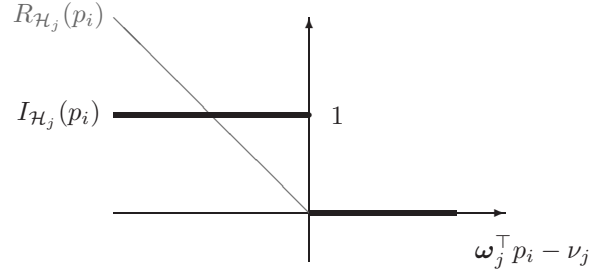
\begin{figure}[t!]
\begin{minipage}[h]{75mm}
\centering
\setlength{\unitlength}{0.65mm}
\begin{picture}(50,85)(0,20)
\newcounter{cms}
\setcounter{cms}{0}
\linethickness{0.075mm}
\put(0,60){\vector(1,0){80}}
\linethickness{0.075mm}
\put(40,50){\vector(0,1){50}}
\put(40,80){\circle*{1}}
\linethickness{0.3mm}
\put(40,60){\color{grayDP} \line(-1,1){40}}
\linethickness{0.6mm}
\put(0,80){\line(1,0){40}}
\put(40,60){\line(1,0){30}}
\put(45,80){\makebox(0,0){{ $1$}}}
\put(-12,80){\makebox(0,0){{ $I_{\mathcal{H}_j}(p_i)$}}}
\put(-12,100){\makebox(0,0){{ \color{grayDP} {$R_{\mathcal{H}_j}(p_i)$}}}}
\put(85,52){\makebox(0,0){{ $\boldsymbol{\omega}_j^\top p_i-\nu_j$}}}
%\multiput(70,60)(0.4,0.8){2}
%{\line(-1,1){10}}

%\put(10,48){\framebox(30,24)}
%\put(11,49){\framebox(28,22)} \put(25,60){\makebox(0,0){{ ${ N}$}}}
%\put(40,60){\vector(1,0){20}} \put(90,60){\vector(1,0){17}}
%\put(113,60){\vector(1,0){17}} \put(110,80){\vector(0,-1){17}}
%\put(110,60){\circle{6}} \put(60,48){\framebox(30,24)}
%\put(75,60){\makebox(0,0){{$ \dfrac{B(q^{-1})}{A(q^{-1})}$}}}
%\put(50,65){\makebox(0,0){{ $x_t$}}} \put(-5,65){\makebox(0,0){{ $u_t$}}}
%\put(95,65){\makebox(0,0){{ $w_t$}}} \put(125,65){\makebox(0,0){{ $y_t$}}}
%\put(115,75){\makebox(0,0){{ $\eta_t$}}} \put(104,64){\makebox(0,0){{\small
%$+$}}} \put(114,66){\makebox(0,0){{\small $+$}}}
\end{picture}
\end{minipage}
\vspace*{-20mm} \caption{Indicator function $I_{\mathcal{H}_j}(p_i)$ (black solid line) and approximate function $R_{\{\mathcal{H}_j\}}(p_i)$ (gray thin line). When  $\boldsymbol{\omega}_j^\top p_i-\nu_j > 0$, $I_{\{\mathcal{H}_j\}}(p_i)$ and $R_{\{\mathcal{H}_j\}}(p_i)$ are overlapped and they are equal to 0.} \label{fig:indfuncapp1}
\end{figure}
%

%
%\begin{figure}
%\begin{center}
%\hspace{-0mm}\includegraphics[bbllx=10mm,bblly=61mm,bburx=179mm,bbury=142mm,scale=0.35]{Fig1.eps}\end{center}
%\caption{Indicator function  $I_{\{\mathcal{H}_j\}}(x_i)$ (solid line) and approximate function $R_{\{\mathcal{H}_j\}}(x_i)$ (dashed line). When  $\omega^{{\tr}}x_i+b^{} < 0$, $I_{\{\mathcal{H}_j\}}(x_i)$ and $R_{\{\mathcal{H}_j\}}(x_i)$ are overlapped and they are equal to 0.}
%\label{fig:indfuncapp1}
%\end{figure}

 Problem \eqref{opt:algv4} is thus relaxed \DP{by}  replacing the indicator functions $I_{\{\mathcal{H}_j\}}(p_i)$ with the convex functions $R_{\{\mathcal{H}_j\}}(p_i)$, i.e.,
  \begin{equation} \label{opt:algv4app}
  \begin{split}
\tilde{\boldsymbol{\omega}}^*_j,\tilde{\nu}^*_j= & \mbox{arg} \min_{\substack{     \boldsymbol{\omega}_j \in \mathbb{R}^n\\ \nu_j \in \mathbb{R} }}
\sum_{i=1}^N R_{\{\mathcal{H}_j\}}(p_i) \\
   & \mbox{s.t. } \\
   & \boldsymbol{\omega}_j \neq 0 \\
   & \mathcal{M} \subseteq \mathcal{H}_j \\
  % & \omega^{{\tr}}x+b^{} \geq 0 \ \ \ \forall x \in \mathcal{S} \\
   & p_i \in \mathcal{L}, \ \ \ i=1,\ldots,N.
  \end{split}
  \end{equation}

%  \begin{equation} \label{opt:algapp}
%  \begin{split}
%&\tilde{\boldsymbol{\omega}}^{}_j,\tilde{\nu^{}}_j=  \mbox{arg}\min_{\substack{     \boldsymbol{\omega}_j \in \mathbb{R}^n \\
%                                            \nu_j \in \mathbb{R} \\
%                                            Q_s \in \mathbb{R}^{s(d),s(d)} }} \sum_{i=1}^N R_{\{\mathcal{H}_j\}}(p_i) \\
%   & \mbox{s.t. } \\
%   &  \boldsymbol{\omega}_j \neq 0 \\
%    & \nu_j - \boldsymbol{\omega}_j \mathbf{x}=    \mathbf{q}_d(\mathbf{x})^\top \mathbf{Q}_0 \mathbf{q}_d(\mathbf{x})+\\
%   -& \sum_{s=1}^m \mathbf{q}_d(\mathbf{x})^\top \mathbf{Q}_s \mathbf{q}_d(\mathbf{x}) h_s(\mathbf{x}), ~~\forall~\mathbf{x} \in \mathbb{R}^n\\
%  & Q_s \succeq 0, \ \ s=0,\ldots,m. \\
%  % & \omega^{{\tr}}x+b^{} \geq 0 \ \ \ \forall x \in \mathcal{S} \\
%   & p_i \in \mathcal{L}, \ \ \ i=1,\ldots,N
%  \end{split}
%  \end{equation}

\begin{theorem} \label{pr:supphype}
If (i) there exists at least one point $p_i$ in the list $\mathcal{L}$ such that $\tilde{\boldsymbol{\omega}}_j^{*^\top} p_i-\tilde{\nu}^*_j < 0$ and (ii) $\tilde{\boldsymbol{\omega}}^*_j,\tilde{\nu}^*_j$ is the optimal solution of problem \eqref{opt:algv4app}, then the hyperplane $\tilde{\boldsymbol{\omega}}_j^{*^\top} \xhyp-\tilde{\nu}^*_j  = 0$ is a supporting hyperplane for the set $\mathcal{M}$.
\end{theorem}
\begin{proof}
Theorem \ref{pr:supphype} is proved by contradiction.  Let  $\tilde{\mathcal{H}}_j^*$ be the half-space defined   as  $\tilde{\mathcal{H}}_j^*: \tilde{\boldsymbol{\omega}}_j^{*^\top} \xhyp-\tilde{\nu}^*_j \leq 0$. Let us suppose that $\tilde{\boldsymbol{\omega}}^*_j,\tilde{\nu}^*_j$ is a feasible solution of problem \eqref{opt:algv4app} such that $\tilde{\boldsymbol{\omega}}_j^{*^\top} \xhyp-\tilde{\nu}^*_j = 0$ is not a supporting hyperplane for $\mathcal{M}$, that is, for some $\varepsilon > 0$,  $\tilde{\mathcal{H}}_j:\tilde{\boldsymbol{\omega}}_j^{*^\top} \xhyp-\tilde{\nu}^*_j+ \varepsilon \leq 0$ for all $\mathbf{x} \in \mathcal{M}$. Let us define $\tilde{\nu}_j$ as $\tilde{\nu}_j=\tilde{\nu}^*_j- \varepsilon$. Note that $\{\tilde{\boldsymbol{\omega}}^*_j,\tilde{\nu}_j\}$ is still a feasible solution of problem \eqref{opt:algv4app} and $\mathcal{\tilde{H}}_j \subseteq \tilde{\mathcal{H}}_j^*$. Let $ V^*=\sum_{i=1}^N R_{\{\tilde{\mathcal{H}}_j^*\}}(p_i)$ be the value of the cost function of Problem \eqref{opt:algv4app} obtained for $\boldsymbol{\omega}^{}=\tilde{\boldsymbol{\omega}}^*_j$ and $\nu^{}=\tilde{\nu}^*_j$.  $R_{\{\tilde{\mathcal{H}}_j^*\}}(p_i)$ is then  given by
\begin{equation}
R_{\{\tilde{\mathcal{H}}_j^*\}}(p_i) =\left\{\!\begin{array}{ll}
                                    -\tilde{\boldsymbol{\omega}}_j^{*^\top} p_i+\tilde{\nu}^*_j & \mbox{if } \tilde{\boldsymbol{\omega}}_j^{*^\top} p_i-\tilde{\nu}^*_j \leq 0\\
                                    0 & \mbox{if } \tilde{\boldsymbol{\omega}}_j^{*^\top} p_i-\tilde{\nu}^*_j > 0
                                    \end{array}
\right.
\end{equation}
Similarly, let $ \tilde{V}=\sum_{i=1}^N R_{\{\mathcal{\tilde{H}}_j\}}(p_i)$ be the value of the cost function of Problem \eqref{opt:algv4app} obtained when $\boldsymbol{\omega}^{}=\tilde{\boldsymbol{\omega}}^*_j$ and $\nu^{}=\tilde{\nu}_j$. The term $R_{\{\mathcal{\tilde{H}}_j\}}(p_i)$ is the given by
\begin{equation}
R_{\{\mathcal{\tilde{H}}_j\}}(p_i)=\left\{\!\begin{array}{ll}
-\tilde{\boldsymbol{\omega}}_j^{*^\top} p_i+\tilde{\nu}_j & \mbox{if } \tilde{\boldsymbol{\omega}}_j^{*^\top} p_i-\tilde{\nu}_j \leq 0\\
                                    0 & \mbox{if } \tilde{\boldsymbol{\omega}}_j^{*^\top} p_i-\tilde{\nu}_j > 0
                                    \end{array}
\right.
\end{equation}
Since $\mathcal{\tilde{H}}_j \subseteq \mathcal{\tilde{H}}_j^*$, then when $R_{\{\tilde{\mathcal{H}}_j^*\}}(p_i) = 0$, also $R_{\{\mathcal{\tilde{H}}_j\}}(p_i)$ is equal to zero. On the other hand, when $R_{\{\tilde{\mathcal{H}}_j^*\}}(p_i)=-\tilde{\boldsymbol{\omega}}_j^{*^\top} p_i+\tilde{\nu}^*_j>0$, then $R_{\{\mathcal{\tilde{H}}_j\}}(p_i)$ can be equal either to zero or to $-\tilde{\boldsymbol{\omega}}_j^{*^\top} p_i+\tilde{\nu}_j=-\tilde{\boldsymbol{\omega}}_j^{*^\top} p_i+\tilde{\nu}^*_j-\varepsilon \leq -\tilde{\boldsymbol{\omega}}_j^{*^\top} p_i+\tilde{\nu}^*_j$. On the basis of the above considerations, it follows:
 \begin{equation}
 \left\{
 \begin{array}{lllllll}
   R_{\{\tilde{\mathcal{H}}_j^*\}}(p_i) &=&  R_{\{\mathcal{\tilde{H}}_j\}}(p_i) & \mbox{ if } \tilde{\boldsymbol{\omega}}_j^{*^\top} p_i-\tilde{\nu}^*_j \geq 0\\
   R_{\{\tilde{\mathcal{H}}_j^*\}}(p_i) &>&  R_{\{\mathcal{\tilde{H}}_j\}}(p_i) & \mbox{ if } \tilde{\boldsymbol{\omega}}_j^{*^\top} p_i-\tilde{\nu}^*_j < 0
 \end{array}
 \right.
 \end{equation}
 Since by hypothesis (i) there exists at least one point $p_i$ in the list $\mathcal{L}$ such that $\tilde{\boldsymbol{\omega}}_j^{*^\top} p_i-\tilde{\nu}^*_j < 0$, it follows that $V^* > \tilde{V}$. Therefore, $\tilde{\boldsymbol{\omega}}_j^{*},\tilde{\nu}^*_j$ is not the optimal solution of problem \eqref{opt:algv4app}. This contradicts hypothesis (ii). \hfill $\blacksquare$
\end{proof}

   Theorem \ref{pr:supphype} has the following interpretation. Among all the half-spaces defined by the normal vector $\tilde{\boldsymbol{\omega}}_j^{*}$ and containing  the set $\mathcal{M}$,  the optimization problem \eqref{opt:algv4app} provides the half-space  $\mathcal{H}_j^*: \tilde{\boldsymbol{\omega}}_j^{*^\top} \xhyp-\tilde{\nu}^*_j\leq 0$ which minimizes the volume of the polytope $\mathcal{B} \cap \mathcal{H}_1^* \cap \ldots \cap \mathcal{H}^*_{j}$, even if the integral  $\int_{\mathcal{B} \cap \mathcal{H}^*_1 \cap \ldots \mathcal{H}^*_{j}}d\mathbf{x}$ is approximated (up to a constant)  by $ \sum_{i=1}^N I_{\{\mathcal{H}^*_j\}}(p_i)$ and the indicator functions $I_{\{\mathcal{H}^*_j\}}(p_i)$ are replaced by the convex functions $R_{\{\mathcal{H}^*_j\}}(p_i)$.\\

\subsection{Handling the constraint $\mathcal{M} \subseteq \mathcal{H}_j$} \label{Sec:HandConst}
The constraints $\mathcal{M} \subseteq \mathcal{H}_j$ can be handled through the SOS-based approach already discussed in Section \ref{subsec:SOS}. Specifically, by introducing a SOS relaxation, Problem  \eqref{opt:algv4app} is replaced by:
  \begin{equation} \label{opt:alg}
  \begin{split}
&\boldsymbol{\omega}^*_j,\nu^*_j=  \mbox{arg}\min_{\substack{     \boldsymbol{\omega}_j \in \mathbb{R}^n \\
                                            \nu_j \in \mathbb{R} \\
                                            \mathbf{Q}_s }} \sum_{i=1}^N R_{\{\mathcal{H}_j\}}(p_i) \\
   & \mbox{s.t. } \\
   &  \boldsymbol{\omega}_j \neq 0 \\
    & \nu_j - \boldsymbol{\omega}_j \mathbf{x}=    \mathbf{q}_\relorder(\xaug)^\top \mathbf{Q}_0 \mathbf{q}_\relorder(\xaug)+\\
   -& \sum_{s=1}^m \mathbf{q}_\relorder(\xaug)^\top \mathbf{Q}_s \mathbf{q}_\relorder(\xaug) h_s(\xaug), ~~\forall~\xaug \in \mathbb{R}^{2n}\\
  & \mathbf{Q}_s \succeq 0, \ \ s=0,\ldots,m. \\
  % & \omega^{{\tr}}x+b^{} \geq 0 \ \ \ \forall x \in \mathcal{S} \\
   & p_i \in \mathcal{L}, \ \ \ i=1,\ldots,N
  \end{split}
  \end{equation}
Note that, as already discussed in Section \ref{subsec:SOS}, the constraint $\nu_j - \boldsymbol{\omega}_j^\top \mathbf{x} \geq 0$ is satisfied for all $\mathbf{x} \in \mathcal{X}$. Therefore, the half-space: $\mathcal{H}_j=\left\{ \xhyp \in \mathbb{R}^n:   \boldsymbol{\omega}_j^\top \xhyp  \leq \nu_j \right\}$ is guaranteed to contain $\mathcal{X}$. Thus, also the set $\mathcal{M}$ is included in $\mathcal{H}_j$.
Finally, note that, in order to deal with the nonconvex constraint $\boldsymbol{\omega}_j \neq 0$ in \eqref{opt:alg}, 
Problem \eqref{opt:alg} can be splitted into the two following SDP problems:
 \begin{subequations} \label{opt:algappSDPsplit}
 \begin{equation}
   \begin{split}
&\overline{\boldsymbol{\omega}}^*_j,\overline{\nu}^*_j=  \mbox{arg}\min_{\substack{     \boldsymbol{\omega}_j \in \mathbb{R}^n \\
                                            \nu_j \in \mathbb{R} \\
                                            \mathbf{Q}_s   }} \sum_{i=1}^N R_{\{\mathcal{H}_j\}}(p_i) \\
   & \mbox{s.t. } \\
   &  \boldsymbol{\omega}_{j,1} =1 \\
    & \nu_j - \boldsymbol{\omega}_j \mathbf{x}=    \mathbf{q}_\relorder(\xaug)^\top \mathbf{Q}_0 \mathbf{q}_\relorder(\xaug)+\\
   -& \sum_{s=1}^m \mathbf{q}_\relorder(\xaug)^\top \mathbf{Q}_s \mathbf{q}_\relorder(\xaug) h_s(\xaug), ~~\forall~\xaug \in \mathbb{R}^{2n}\\
  & \mathbf{Q}_s \succeq 0, \ \ s=0,\ldots,m. \\
  % & \omega^{{\tr}}x+b^{} \geq 0 \ \ \ \forall x \in \mathcal{S} \\
   & p_i \in \mathcal{L}, \ \ \ i=1,\ldots,N
  \end{split}
 \end{equation}
  \begin{equation}
   \begin{split}
&\underline{\boldsymbol{\omega}}^*_j,\underline{\nu}^*_j=  \mbox{arg}\min_{\substack{     \boldsymbol{\omega}_j \in \mathbb{R}^n \\
                                            \nu_j \in \mathbb{R} \\
                                            \mathbf{Q}_s  }} \sum_{i=1}^N R_{\{\mathcal{H}_j\}}(p_i) \\
   & \mbox{s.t. } \\
   &  \boldsymbol{\omega}_{j,1} =-1 \\
    & \nu_j - \boldsymbol{\omega}_j \mathbf{x}=    \mathbf{q}_\relorder(\xaug)^\top \mathbf{Q}_0 \mathbf{q}_\relorder(\xaug)+\\
   -& \sum_{s=1}^m \mathbf{q}_\relorder(\xaug)^\top \mathbf{Q}_s \mathbf{q}_\relorder(\xaug) h_s(\xaug), ~~\forall~\xaug \in \mathbb{R}^{2n}\\
  & \mathbf{Q}_s \succeq 0, \ \ s=0,\ldots,m. \\
  % & \omega^{{\tr}}x+b^{} \geq 0 \ \ \ \forall x \in \mathcal{S} \\
   & p_i \in \mathcal{L}, \ \ \ i=1,\ldots,N
  \end{split}
 \end{equation}
 \end{subequations}
 with $\boldsymbol{\omega}_{j,1}$ denoting the first component of vector $\boldsymbol{\omega}_j$. The optimizer $\{\boldsymbol{\omega}^*_j,\nu^*_j\}$ of Problem  \eqref{opt:alg} is the given by the pair  $\{\overline{\boldsymbol{\omega}}^*_j,\overline{\nu}^*_j\}$ or  $\{\underline{\boldsymbol{\omega}}^*_j,\underline{\nu}^*_j\}$ that provides the minimum value of the objective function $\sum_{i=1}^N R_{\{\mathcal{H}_j\}}(p_i)$.

\begin{remark}
 \DP{For a fixed degree $2\relorder$ of the SOS polynomials, the number of optimization variables of Problems \eqref{opt:algappSDPsplit} increases  polynomially with the state dimension $n$ and linearly with the number $m$ of constraints $h_s(\xaug)$ defining the set $\mathcal{X}$. Specifically, the number of optimization variables of Problem \eqref{opt:algappSDPsplit} is $O(mn^{2\relorder})$. In fact, the number of free decision variables in the matrices $\mathbf{Q}_s$ (with $s=0,\ldots,m$) is $\dfrac{{2n+\relorder \choose \relorder}\left(1+{2n+\relorder \choose \relorder}\right)}{2}=O(n^{2\relorder})$. On the other hand, for a fixed $n$, the size of the matrices $\mathbf{Q}_s$ increases exponentially with the degree $2\relorder$ of the SOS polynomials. In order not to obtain too conservative results, practical experience of the authors  suggests to take   $\relorder\geq \lceil \frac{d}{2} \rceil+1$, where $\lceil \cdot \rceil$ denotes the ceiling operator. We  remind that $d$ is the degree of the considered polynomial system in \eqref{eq:setmembersys}. Roughly speaking, because of memory requirement issues, the relaxed SDP problems \eqref{opt:algappSDPsplit} can be solved in commercial workstations and with general purpose SDP solvers like \emph{SeDuMi} in case of polynomial systems with $4$ state variables and of degree   $d$  not greater than $6$. Systems with more state variables can be considered in case of smaller values of $d$. Similarly, systems of higher degree can be considered  in case of a smaller number of state variables.}
 \end{remark}

 \begin{remark}
\DP{ As already discussed, Algorithm 2 computes, at each iteration, an half-space $\mathcal{H}_j: \boldsymbol{\omega}_j^{^\top} \xaug-\nu_j\leq 0$ containing the set $\mathcal{X}$ (thus also $\mathcal{M}$), i.e.,
 \begin{equation} \label{eqn:robCOnst}
 \boldsymbol{\omega}_j^{^\top} \xaug-\nu_j\leq 0 \ \ \ \ \forall\ \xaug \in \mathcal{X}.
 \end{equation}
 The parameters $ \boldsymbol{\omega}_j$ and $\nu_j$ are then computed by solving Problem \eqref{opt:alg}, and replacing the robust constraint \eqref{eqn:robCOnst} with a SOS constraint (see Problem \eqref{opt:alg}). Note that  the same principles of Algorithm 2 and of the SOS-based relaxation  discussed in  this section can be used to compute, instead of  an half-space $\mathcal{H}_j$, a more complex semialgebraic set (e.g., an ellipsoid) described by the polynomial inequality:
  \begin{equation} \label{eqn:robCOnst2}
 \boldsymbol{\omega}{^\top} \textbf{q}(\xaug)\leq 0 \ \ \ \ \forall\ \xaug \in \mathcal{X},
 \end{equation}
 with $\textbf{q}(\xaug)$ being a vector of monomials in the variable $\xaug$. The parameters  $\boldsymbol{\omega}$ can be then computed by properly modifying the SOS-relaxed  Problem \eqref{opt:alg}. For instance, in case we are interested in computing an ellipsoidal outer approximation of  $\mathcal{X}$, the  function  $\boldsymbol{\omega}{^\top} \textbf{q}(\xaug)$ should have a quadratic form, and its Hessian should be enforced to be positive definite.}
  \end{remark}

 %Indeed,  the free decision variables in the matrices $\mathbf{Q}_s$ are the coefficients of the $2d$-degree SOS polynomial $\mathbf{q}_d(\xaug)^\top \mathbf{Q}_s \mathbf{q}_d(\xaug)$ in the variable $\xaug \in \mathbb{R}^{2n}$.

  %where the constraint $z_0 \leq 0$ has been added in order to guarantee that the half-space $\mathcal{H}_j: \boldsymbol{\omega}_j^T \int \mathbf{x}dP(\mathbf{x}|\mathbf{y})-\nu_j \leq 0$ contains $\Omega_2$.

% In the following, the subscript $j$ will be dropped for simplicity of notation.
%\subsection{On the complexity of Problems \eqref{opt:algappSDPsplit}}
%The number of optimization variables of the  SDP problems \eqref{opt:algappSDPsplit} is $$ can be solved in polynomial time through interior-points methods

\begin{example}
 \label{ex:3}
 Let us continue with Example \ref{ex:1}.  Fig. \ref{fig:Alg2}
 shows the polytope obtained by applying Algorithm 2 solving Problems \eqref{opt:algappSDPsplit} instead of the
 nonconvex optimization in \DP{A2.2.}  The SDP   Problems \eqref{opt:algappSDPsplit} are solved for a degree of the SOS polynomials equal to $2\relorder=4$. The solution is a polytope $\mathcal{S}$ that outer-bounds $\mathcal{X}_1$.
 It can be observed that because of the approximations introduced (SOS and the approximation of the indicator functions), which are necessary to efficiently solve the optimizations, the half-spaces bounding $\mathcal{X}_1$ are not tangent to it and the computed region $\mathcal{S}$ still include
 two black points. Therefore, the computed polytope is not the minimum-volume polytope. However, it is already a very good outer-approximation of it. In the next section, we describe a further refinement of Algorithm 2 aiming
 to computing a tighter polytope $\mathcal{S}$.
  According to the steps A1.1.3 and A1.1.4 of Algorithm 1, we outer-approximate $\mathcal{M}$ (and so $\mathcal{X}_1$) with $\mathcal{S}$.
 At the next time step  ($k=2$) of the set-membership filter, we repeat the procedure to compute a new polytope
 outer-bounding  $\mathcal{X}_2$. The difference is now that instead of $h_1(\cdot)$ in \eqref{eq:defhs},
 we have the $9$ linear inequalities that define the polytope in  Fig. \ref{fig:Alg2}.
 This procedure is repeated recursively in time.
\begin{figure}[h]
        \centering
                     \includegraphics[width=0.3\textwidth]{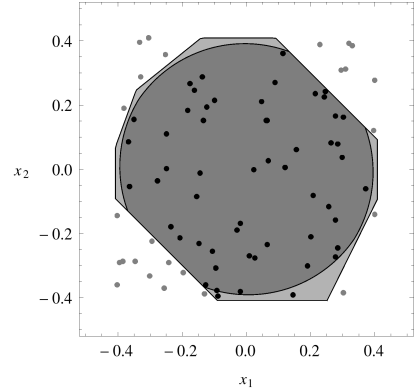}
          \caption{Final polytope after running Algorithm 2.}\label{fig:Alg2}
\end{figure}

\end{example}

%\section{Additional refinements}
\subsection{Refinement of the polytope $\mathcal{S}$}

Summarizing, an approximate solution of the robust optimization problem \eqref{opt:algv4} is   computed by solving the convex SDP  problems \eqref{opt:algappSDPsplit}, and, on the basis of Algorithm 2, the polytopic-outer approximation $\mathcal{S}$ of the set $\mathcal{M}$ is then defined as $\mathcal{S}= \mathcal{B} \cap \mathcal{H}_1^{} \cap \ldots \cap \mathcal{H}_{J}^{}$. \\
Note that, in solving \eqref{opt:algappSDPsplit} instead of \eqref{opt:algv4}, two different sources of approximation are introduced:
\begin{itemize}
  \item Approximation of the indicator functions $I_{\{\mathcal{H}_j\}}(p_i)$ with the convex functions $R_{\{\mathcal{H}_j\}}(p_i)$ (see Fig. \ref{fig:indfuncapp1});
  \item Approximation of the robust constraint $\nu -\boldsymbol{\omega}^{{\top}}\mathbf{x}  \geq 0 \ \ \forall \mathbf{x} \in \mathcal{X}$ with the convex conservative constraint $\nu -\boldsymbol{\omega}^{{\top}}\mathbf{x} =\sigma_0(\xaug)-\sum_{s=1}^m \sigma_s(\xaug) h_s(\xaug)$.
\end{itemize}
The latter source of approximation can be reduced by increasing the degree $2\relorder$ of the SOS polynomials. In fact, as already discussed in Section \ref{subsec:SOS}, according to the  \emph{Putinar's Positivstellensatz} each function $\nu -\boldsymbol{\omega}^{{\top}}\mathbf{x}$ such that $\nu -\boldsymbol{\omega}^{{\top}}\mathbf{x} \geq 0 \ \forall \mathbf{x} \in \mathcal{X}$ can be written as $ \nu -\boldsymbol{\omega}^{{\top}}\mathbf{x}=\sigma_0(\xaug)-\sum_{s=1}^m \sigma_s(\xaug) h_s(\xaug)$ provided that the degree of the SOS polynomials $\sigma_0, \sigma_1, \ldots, \sigma_m$ is large enough.
%On the other hand, although relaxation of  $\ell_0$-seminorm with  $\ell_1$ norm has been efficiently used in identification and control community (see Remark \ref{remarkl0}), there is no theoretical result concerning the accuracy of the approximation of the objective function of problem \eqref{opt:alg} with the convex functional of problem \eqref{opt:algapp}.
On the other hand, there is no theoretical result concerning the accuracy of the approximation of the indicator functions in Problem \eqref{opt:algv4} with the convex functions $R_{\{\mathcal{H}_j\}}(p_i)$  appearing in Problem \eqref{opt:algappSDPsplit}.
Because of that, the polytope  $\mathcal{S}$ obtained by solving convex problems \eqref{opt:algappSDPsplit} (for $j=1,\ldots,J$) is not guaranteed to minimize the original nonconvex optimization problem \eqref{appr:problinhullappr}.  Algorithm 3 can then be used to   refine the polytopic outer approximation $\mathcal{S}$ provided by Algorithm 2.

\begin{algorithm} \emph{Algorithm 3: Refinement of the polytope $\mathcal{S}$} \label{alg2}

 \noindent [\emph{input}] Sequence of the random points $p_i$ provided as input of Algorithm 2 and such that $p_i \in \mathcal{S}$. Let $\tilde{N}$ be the number of points $p_i$ belonging to $\mathcal{S}$.
\begin{description}
\item [A3.1 ] \;  $\mathcal{S}^{*} \leftarrow \mathcal{S}$
  \item[A3.2 ] \; \; \; \; \; for $\ i=1:\tilde{N}$
  \begin{description}
  \item[A3.2.1  ] \; \; \;  \; \; Compute the solution of the following optimization problem
                    \begin{equation} \label{opt:alg22}
  \begin{split}
  \boldsymbol{\omega}^*_i,\nu^*_i= & \mbox{arg}\min_{\begin{array}{c}
                                            \omega \in \mathbb{R}^n \\
                                            \nu \in \mathbb{R}
                                          \end{array}} -\boldsymbol{\omega}^{{\top}}p_i+\nu^{} \\
   & \mbox{s.t. } \\
   & \boldsymbol{\omega} \neq 0 \\
   & \nu- \boldsymbol{\omega}^{\top}   \mathbf{x} \geq 0 \ \ \forall \mathbf{x} \in \mathcal{X}.\\
  \end{split}
  \end{equation}
 % \item[A4.2.3] \; \; \; \; \; if $\boldsymbol{\omega}_i^{*^\top}p_i-\nu_i^* > 0$ (that is, the points $p_i$ does not belong to the half-space defined as $\mathcal{H}_i: \boldsymbol{\omega}_i^{*^\top}\tilde{\mathbf{x}}-\nu_i^* \leq 0$), then  $\mathcal{S}^{*} \leftarrow \mathcal{S}^{*} \cap \mathcal{H}_i$.
   \item[A3.2.2] \; \; \; \; \; $\mathcal{S}^{*} \leftarrow \mathcal{S}^{*} \cap \mathcal{H}_i$.
   % \item[A3.2.3] \; \; \; \; \; end for \\
\end{description}
\end{description}
\noindent  [output] Polytope $\mathcal{S}^{*}$.
\end{algorithm}

\DP{The main principle of Algorithm 3 is to process, one by one, all the points belonging to the  polytopic outer-approximation $\mathcal{S}$ initially given by Algorithm 2. For each of such points $p_i$,  an half-space $ \mathcal{H}_i: \boldsymbol{\omega}_i^{*^\top}\tilde{\mathbf{x}}-\nu_i^* \leq 0$ including the set $\mathcal{X}$ (i.e., $\mathcal{X} \subseteq \mathcal{H}_i$) and at the same not containing the point $p_i$ (i.e., $p_i \not\in \mathcal{H}_i$, or equivalently  $-\boldsymbol{\omega}_i^{*^\top}p_i+\nu_i^* < 0$) is seeked. In this way, all the points $p_i$ which do not belong to the minimum volume polytopic outer approximation of $\mathcal{X}$ are discarded. Thus, a tighter (but more complex) polytopic outer approximation of $\mathcal{X}$ is obtained.}

%%The main idea of Algorithm 3  is to compute, for each point $p_i$ belonging to the initial polytopic outer-approximation $\mathcal{S}$, an half-space $ \mathcal{H}_i: \boldsymbol{\omega}_i^{*^\top}\tilde{\mathbf{x}}-\nu_i^* \leq 0$ such that $\mathcal{M} \subseteq \mathcal{H}_i$  and $p_i \not\in \mathcal{H}_i$ (i.e., $-\boldsymbol{\omega}_i^{*^\top}p_i+\nu_i^* < 0$).  %If such an halfspace exists, then the new inequality constraint $\mathcal{H}_i: \omega^{(i,*)^{\top}}\mathbf{x}-\nu^{(i,*)} \leq 0$ is added in the definition of the polytope  $\mathcal{S}^{*}$.
An important feature enjoyed by the refined polytope  $\mathcal{S}^{*}$ is given by the following theorem.\\

\begin{theorem} \label{the:ref}
The polytope $\mathcal{S}^{*}$ computed with  Algorithm 3 is a global minimizer of problem \eqref{appr:problinhullappr}.
\end{theorem}
\begin{proof}
Let  $\tilde{\mathcal{S}}$  be a polytope  belonging to the set of feasibility of problem \eqref{appr:problinhullappr} (i.e., $\mathcal{M} \subseteq \tilde{\mathcal{S}}$) which does not minimize  \eqref{appr:problinhullappr}. This means that there exists a polytope $\tilde{\tilde{\mathcal{S}}}$ such that $\mathcal{M} \subseteq \tilde{\tilde{\mathcal{S}}} \subseteq  \tilde{\mathcal{S}} $ and a point $\bar{p}$ given as input of Algorithm 2 such that:  $\bar{p} \in \tilde{\mathcal{S}}$ and $\bar{p} \not\in \tilde{\tilde{\mathcal{S}}}$. Thus, for $p_i=\bar{p}$, the optimal solution $\{\boldsymbol{\omega}^*_i,\nu^*_i\}$ of Problem \eqref{opt:alg22} is such that $\boldsymbol{\omega}_i^{*^\top}p_i-\nu_i^* > 0$. Let $\mathcal{H}_i$ be the half-space defined as $\mathcal{H}_i: \boldsymbol{\omega}_i^{*^\top} \mathbf{x}-\nu_i^* \leq 0$. Obviously, $\bar{p} \not\in \mathcal{H}_i$. Besides,  the output $\mathcal{S}^{*}$ of Algorithm 3  is contained in the hyperspace $\mathcal{H}_i$. Therefore, since $\bar{p} \not\in \mathcal{H}_i$ and $\mathcal{S}^{*} \subseteq\mathcal{H}_i$, it follows that  the point $\bar{p} \not\in \mathcal{S}^{*}$. Then, a polytope $\tilde{\mathcal{S}}$ that does not minimize the optimization problem \eqref{appr:problinhullappr}  can not be the output of Algorithm 3. \hfill $\blacksquare$
\end{proof}

\DP{Theorem \ref{the:ref} mainly says that there exists no polytope including $\mathcal{M}$ and containing less randomly generated  points $p_i$ than  $\mathcal{S}^{*}$.   However, it is worth remarking that only an approximated solution  of Problem \eqref{opt:alg22} can be computed,  as the robust constraint $\nu- \boldsymbol{\omega}^{\top}   \mathbf{x} \geq 0  \ \ \forall \mathbf{x} \in \mathcal{X}$ appearing in \eqref{opt:alg22} has to be handled with the SOS-based  techniques described in the previous section. Thus, conservativeness could be added at this step. Therefore, the main interpretation to be given to Theorem \ref{the:ref}  is that    Algorithm 3 cancels the effect of approximating the indicator function $I_{\{\mathcal{H}_j\}}(p_i)$ with the convex function $R_{\{\mathcal{H}_j\}}(p_i)$.}

\begin{example}
 \label{ex:4}
Let us continue with Example \ref{ex:1}. Fig. \ref{fig:ex1TreuSUS} shows the computed polytope $\mathcal{S}^*_1$, along with the true state uncertainty set $\mathcal{X}_1$. %The polytope has $80$ faces.
The   CPU taken by the proposed algorithm to compute the $54$ hyper-spaces that define the polytope  $\mathcal{S}^*_1$  is about $830$ seconds. However, only $80$ out of $830$ seconds are spent by the  solver \emph{SeDuMi} to solve $108$ (i.e., $54 \times 2$) SDP problems of the type \eqref{opt:algappSDPsplit}. The other $750$ seconds are required by the \emph{SOStools} interface to formulate, $108$ times, the SDP problems \eqref{opt:algappSDPsplit} in the format used by  SeDuMi. Therefore, the computational time required to compute the polytope $\mathcal{S}^*_1$ can be drastically reduced not only by using more efficient SDP solvers, but also directly formulating the SDP problems \eqref{opt:algappSDPsplit} in the format required by the used SDP solver.

%   Then, $N=80$ random points uniformly distributed in $\mathcal{B}_1$ are generated to approximate the volume of  the polytope $\mathcal{S}^*_1$ (as described in Section \ref{susec:area}). Finally, the SDP   Problems \eqref{opt:algappSDPsplit} are solved for a degree of the SOS polynomials equal to $2d=4$.  No upper bound on the number of halfspaces defining $\mathcal{S}^*_1$ is imposed. The   CPU taken by the proposed algorithm to compute the polytope  $\mathcal{S}_1$  is about 120 seconds. Fig. \ref{fig:ex1TreuSUS} shows the computed polytope $\mathcal{S}^*_1$, along with with the true state uncertainty set $\mathcal{X}_1$.
%Results in Fig. \ref{fig:ex1TreuSUS} show that: (i) the computed polytope $\mathcal{S}^*_1$  contains the true set $\mathcal{X}_1$, as expected; (ii) the hyperplanes describing the poltyope  $\mathcal{S}_1$ are supporting hyperplanes for the set $\mathcal{X}_1$ (i.e., although the robust constraint $\displaystyle z_0-\omega^{{\top}}\mathbf{x}+\nu^{} \geq 0 \ \ \ \forall \mathbf{x} \in \mathcal{X}_1$ has been approximated with the convex constraint $\displaystyle \sigma_0-\sum_{s=1}^m \sigma_s h_s(x)$,   Proposition \ref{pr:supphype} still holds in the example).

\begin{figure}[!t]
\centerline{
\includegraphics[scale=0.5]{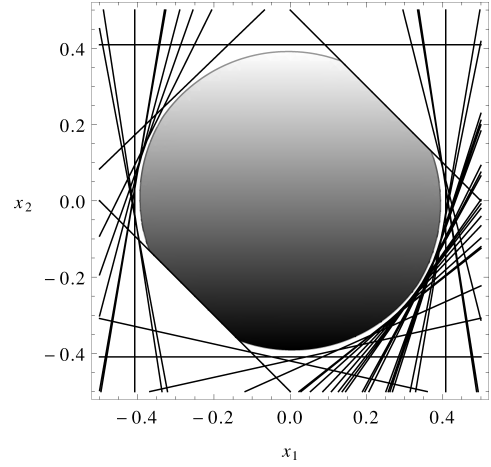}
} \caption{Exampe 1: hyperplanes defining the polytope $\mathcal{S}^*_1$ (black lines) and true state uncertainty set $\mathcal{X}_1$ (gray region).}
\label{fig:ex1TreuSUS}
\end{figure}

 \end{example}

\section{Numerical examples} \label{Sec:Exmple}
Let us consider the discrete-time Lotka Volterra prey-predator model \cite{2013Raj} described by the difference equations:
{  \begin{equation}
\label{eq:Volterra}
\begin{array}[pos]{rcl}
 x_1\!(k)&\!\!=\!\!& x_1(\!k\!-\!1\!)\!\left(r\!+\!1\!-\!rx_1(k\!-\!1)\!-\!bx_2(k\!-\!1)\!\right)\!+\!w_1(\!k\!-\!1\!),\\
   %x_1(k)&\!=\!& (r+1)x_1(k\!-\!1)\!-\!rx^2_1(k\!-\!1)\!-\!bx_1(k\!-\!1)x_2(k\!-\!1)\!+\!w_1(\!k\!-\!1\!),\\
   x_2\!(k)&\!\!=\!\!& cx_1(k\!-\!1)x_2(k\!-\!1)+(1-d)x_2(k\!-\!1)\!+\!w_2(\!k\!-\!1\!),
\end{array}
\end{equation}}
where $x_1(k)$ and $x_2(k)$ denote the prey and the predator population size, respectively. In the example, the following values of the parameters are considered: $r=0.25$, $b=0.95$, $c=1.1$ and  $d=0.55$. The observed output is the sum of the population of the prey and predator densities, i.e.,
\begin{equation}
\mathbf{y}(k)=x_1(k)+x_2(k)+\mathbf{v}(k),
\end{equation}
where the measurement noise $\mathbf{v}(k)$ is bounded and such that $\|\mathbf{v}(k)\|_{\infty} \leq 0.05$. The initial prey and predator sizes $\mathbf{x}(0)=\left[ x_1(0) \ \ x_2(0)\right]^{\top}$ are known to belong to the box $\displaystyle \mathcal{X}_0=\left[ 0.28 \ \ 0.32\right]\times\left[ 0.78 \ \ 0.82\right]$ and the noise process $\mathbf{w}(k)=[w_1(k) \ \ w_2(k)]^\top$ is bounded by $\|\mathbf{w}(k)\|_{\infty} \leq 0.001$. The data are obtained by simulating the model with initial conditions $x_1(0)=0.8$ and $x_2(0)=0.3$, and by corrupting the output observations with a random noise $\mathbf{v}(k)$ uniformly distributed within the interval $[-0.05 \ \ 0.05]$.

Polytopic outer approximations $\mathcal{S}^*_k$ of the state uncertainty sets $\mathcal{X}_k$ (with $k=1,\ldots,40$) are computed through Algorithm 2. $N=20$ random points  are used to approximate the volume of  the polytope $\mathcal{S}^*_k$ (as described in Section \ref{susec:area}).  In order to limit the complexity in the description of the polytopes $\mathcal{S}^*_k$, the maximum number of halfspaces describing $\mathcal{S}^*_k$ is set to $8$. This means that Algorithm 2 is stopped after at most $4$ iterations (we remind that the initial outer-bounding box $\mathcal{B}_k$ is already described by $4$ half-spaces). When the output of Algorithm 2 is a polytope $\mathcal{S}^*_k$ described by less than $8$ half-spaces, Algorithms 3 is  used to refine the polytopic outer approximation  $\mathcal{S}^*_k$.
  Fig. \ref{fig:setP} shows the computed polytopes   $\mathcal{S}^*_k$ outer approximating  the state uncertainty sets $\mathcal{X}_k$ (with $k=1,\ldots,40$), along with the true state trajectory.  The \emph{Hybrid toolbox} \cite{HybTBX} has been used to plot the polytopes in Fig. \ref{fig:setP}. The average CPU time required  to compute a polytope  $\mathcal{S}^*_k$ is $28$ seconds (not including the time required by the \emph{SOStools} interface to formulate the SDP problems \eqref{opt:algappSDPsplit} in the format used by the solver \emph{SeDuMi}). For the sake of comparison, Fig. \ref{fig:setB} shows the outer-bounding approximations of the state uncertainty sets $\mathcal{X}_k$ when boxes, instead of polytopes, are propagated over time. For a better comparison, in Fig. \ref{fig:setstate} the bounds on the time-trajectory of each state variable obtained by propagating boxes and polytopes are plotted. The obtained results show that, as expected, propagating polytopic uncertainty sets instead of boxes provides a more accurate state estimation. Finally, we would like to remark that a small uncertainty on the noise process is assumed (i.e., $\|\mathbf{w}(k)\|_{\infty} \leq 0.001$) since, for larger bounds on $\|\mathbf{w}(k)\|_{\infty}$, it would not be possible to   clearly visualize the uncertainty boxes in Fig. \ref{fig:setB}.

\begin{figure}[!t]
\centerline{
\includegraphics[scale=1]{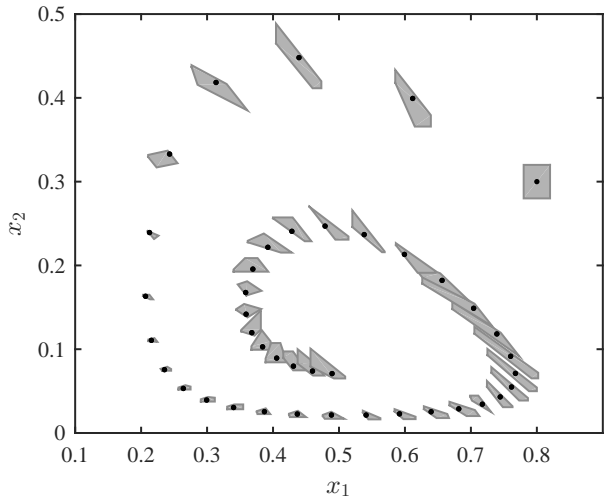}
} \caption{Example 2: outer-bounding polytopes (gray) and true state trajectory (black dots).} \vspace{2mm}
\label{fig:setP}
\end{figure}

\begin{figure}[!t]
\centerline{
\includegraphics[scale=1]{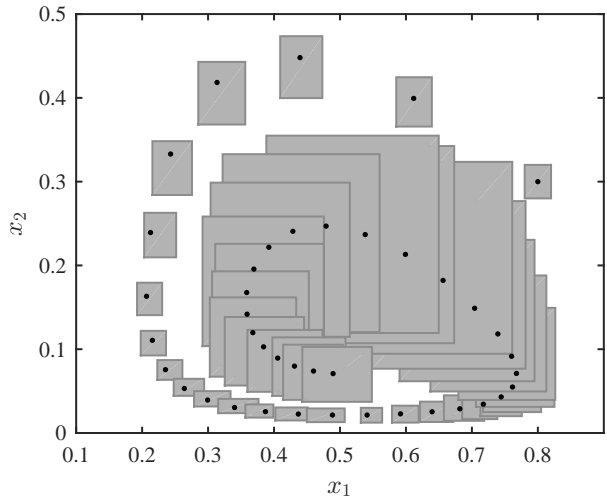}
} \caption{Example 2:  outer-bounding boxes (gray) and true state trajectory (black dots).} \vspace{2mm}
\label{fig:setB}
\end{figure}

\begin{figure}[!t]
\centerline{
\includegraphics[scale=1]{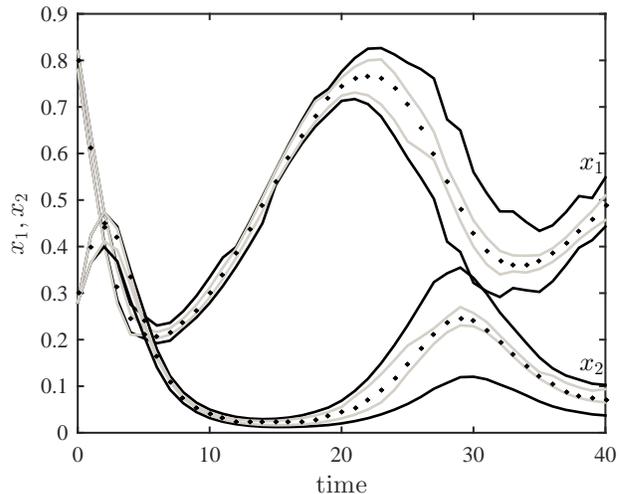}
} \caption{Example 2: bounds on state trajectories obtained by propagating boxes (black line); bounds on state trajectories obtained by propagating polytopes (gray line); true state trajectory (black dots).}
\label{fig:setstate}
\end{figure}

\section{Conclusions}
In this paper we have shown that set-membership estimation can  be equivalently formulated in a probabilistic setting by employing sets of probability measures.
Inferences in set-membership estimation  are thus carried out by computing expectations \DP{with respect to}  the updated
set of probability measures $\mathcal{P}$, as in the probabilistic case, and they can be formulated as a semi-infinite linear programming problem.
We have further shown that, if the nonlinearities in the measurement and process equations are polynomial and if the bounding sets for initial state, process and measurement noises are described by polynomial inequalities, then an approximation of this semi-infinite linear programming problem can be obtained by using the theory of sum-of-squares polynomial optimization. We have finally derived a procedure to compute a polytopic outer-approximation of the true membership-set, by computing the  minimum-volume polytope
that outer-bounds the set that includes all the means computed \DP{with respect to}  $\mathcal{P}$.
\DP{It is worth remarking that the set-membership filtering approach discussed in the paper can be extended  to handle 
noise-corrupted input signal observations and uncertainty in the model parameters, provided that the corresponding state 
uncertainty set $\mathcal{X}_k$ remains a semi-algebraic set.} As future works, we aim first to speed up the proposed 
state estimation algorithm in order to be able to use it in real-time applications in systems with fast dynamics.
    \DP{To this aim, dedicated numerical algorithms, written in  Fortran and C++, for solving the formulated SDP optimization problems will be developed. Furthermore,}  the SDP problems will be directly formulated in the format required by the SDP solver, thus avoiding the use of interfaces like \emph{SOStools}. An open source toolbox will be then released.
Second, by exploiting the  probabilistic interpretation of set-membership estimation, we plan to reformulate it using the theory of moments developed by Lasserre. This will allow us to ground totally  set-membership estimation in the realm of the probabilistic setting, which will give us
the possibility of   combining the two approaches in order to obtain hybrid filters, i.e., filters that include both classical probabilistic uncertainties and
set-membership uncertainties.

\bibliographystyle{ieeetr}        % Include this if you use bibtex
\bibliography{FIlteringSMbib}           % and a bib file to produce the

 \end{document}